\documentclass{amsart}

\usepackage{amsmath}
\usepackage{amsfonts}
\usepackage{amssymb}
\usepackage{amsthm}
\usepackage{mathtools}
\usepackage{tikz}
\usepackage{tikz-cd}
\usepackage{hyperref}
\usepackage{enumitem}
\usepackage{graphicx}
\usepackage{subcaption}

\usetikzlibrary{arrows}
\usetikzlibrary{decorations.pathreplacing}
\usetikzlibrary{calligraphy}
\usetikzlibrary{external}

\renewcommand{\H}{\mathrm{H}}

\tikzset{knot/.style={black,very thick,preaction={draw,white,line width=5pt}}}

\colorlet{dark green}{green!50!black}

\DeclareMathOperator{\id}{id}

\newcommand{\F}{\ensuremath{\mathbb{F}} }
\newcommand{\N}{\ensuremath{\mathbb{N}} }
\newcommand{\Z}{\ensuremath{\mathbb{Z}} }

\newcommand{\R}{\ensuremath{\mathbb{R}} }

\newcommand{\B}{\ensuremath{\mathbb{B}} }

\newcommand{\gr}{\mathrm{gr}}

\newcommand{\Br}{\mathrm{Br}}

\newtheorem{theorem}{Theorem}[section]

\newtheorem{proposition}[theorem]{Proposition}

\newtheorem{lemma}[theorem]{Lemma}

\theoremstyle{definition}
\newtheorem{definition}[theorem]{Definition}

\newtheorem{claim}[theorem]{Claim}
\theoremstyle{remark}
\newtheorem*{remark}{Remark}
\newtheorem{example}[theorem]{Example}

\newcommand{\Ras}{\mathrm{Ras}}
\newcommand{\Rou}{\mathrm{Rou}}
\newcommand{\Kho}{\mathrm{Kho}}

\title{Reduced Khovanov-Rozansky homology}
\author{David Popovi\'c}
\address{Department of Mathematics\\ University of California, Los Angeles}
\email{dpopovic@math.ucla.edu}
\date{June 2026}

\begin{document}
\begin{abstract}
We introduce a variant of reduced Khovanov-Rozansky homology defined in terms of the variables $\alpha_1, \dots, \alpha_{n-1}$ that correspond to the regions between the strands. We rigorously reestablish several folklore results in terms of our framework, including the equivalence between Khovanov's and Rasmussen's approaches. Furthermore, we show that Rasmussen's chain complex is a free resolution of the Rouquier chain complex.
\end{abstract}
\maketitle
\section{Introduction}
The Alexander polynomial \cite{alexander1928topological} occupies a special place in low-dimensional topology as the first, and for many decades only, polynomial link invariant. It was Alexander who first observed that the polynomial satisfies the skein relation, but the significance of this realization was only revealed 70 years later with the discovery of the Jones polynomial \cite{jones1997polynomial}. This development sparked a frenzy of activity in pursuit of the most general link invariant that would generalize both the Alexander and Jones polynomials. It was discovered simultaneously by 8 researchers \cite{freyd1985new, przytycki2016invariants} and is commonly referred to as the HOMFLY or HOMFLYPT polynomial due to their last names' initials.

\vspace{1em}

In the $21^\text{st}$ century, the research in the field has shifted from studying the polynomials directly to categorifying them into knot homology theories. Following Khovanov's categorification of the Jones polynomial \cite{khovanov2000categorification} and Ozsv\'ath-Szabo's categorification of the Alexander polynomial \cite{ozsvath2004holomorphicKnots}, it became evident that the HOMFLY polynomial deserved a categorification as well. The existence and some structural predictions about such a homology theory were first conjectured in \cite{dunfield2006superpolynomial} and later constructed by Khovanov-Rozansky in \cite{khovanov2008matrix} based on their earlier related work \cite{khovanov2004matrix}. There exist at least three equivalent approaches to this theory:
\begin{enumerate}
    \item Khovanov and Rozansky's original construction \cite{khovanov2008matrix} using the language of matrix factorizations,
    \item Rasmussen's reformulation \cite{rasmussen2016some} using the language of double chain complexes, which added several important features to theory, and 
    \item Khovanov's reformulation \cite{khovanov2007triply} of the theory in terms of Hochschild homology of Soergel bimodules.
\end{enumerate}
The theory is most commonly known as the Khovanov-Rozansky homology, named after its
originators, or the HOMFLY homology. The three approaches are equivalent, but the precise statements and rigorous proofs of results to that effect, while understood by the experts, do not always appear in the literature. Moreover, every justification of equivalence between Rasmussen's and Khovanov's constructions passes through Khovanov and Rozansky's original construction using matrix factorizations. This is not ideal, because it makes it difficult to simultaneously study certain aspects of Khovanov-Rozansky homology that only exist in Rasmussen's or only in Khovanov's setting. For example, such phenomena include spectral sequences to $\mathfrak{sl}_k$ homology \cite{rasmussen2016some} and $y$-ification \cite{gorsky2022hilbert}. One of the results of this paper is a detailed and rigorous explanation of the direct equivalence between Rasmussen's and Khovanov's constructions.

For technical reasons, we work over the field $\F$ with $\mathrm{char}(\F) \neq 2$.
\begin{theorem}\label{thm:equivalence Ras and Kho homology level}
    Let $K$ be a knot and let $\widehat{\mathit{H}}_{\Kho}(K)$ and $\widehat{\mathit{H}}_{\Ras}(K)$ denote the reduced Khovanov-Rozansky homologies calculated using Khovanov's and Rasmussen's constructions respectively. Then $\widehat{\mathit{H}}_{\Kho}(K) \cong \widehat{\mathit{H}}_{\Ras}(K)$ as $\Z\oplus\Z\oplus\Z$-graded vector spaces over $\F$.
\end{theorem}
To establish this isomorphism, we introduce a somewhat original reformulation of reduced Khovanov-Rozansky homology -- both Khovanov's and Rasmussen's constructions are primarily intended for \emph{middle} or \emph{unreduced} Khovanov-Rozansky homology and as such assign variables $X_1, \dots, X_n$ to the strands of the braid. Passing to the reduced homology is possible by setting $X_i = 0$ for some $i$ at an appropriate point in the construction. However, this is not very elegant since it breaks the symmetry between the $X_i$ which at least \emph{a priori} raises some questions about the well-definedness of the resulting construction. Perhaps with this in mind, Khovanov \cite{khovanov2007triply} suggested replacing the variables $X_1, \dots, X_n$ with the variables $X_2-X_1, \dots, X_n-X_{n-1}$, but the program has never been carried out. Following his approach, we set $\alpha_i = X_{i+1}-X_i$ and interpret them as marking the regions between the strands instead, as in Figure \ref{fig:alphas_instead_of_xs}. This maintains the symmetry in the formulas, which frequently look much nicer, are more easily motivated and are more pleasant to calculate with.
\begin{figure}[t]
    \centering
    \begin{tikzpicture}
        \draw[red, ->, >=latex] (1.5, -0.25) -- (2, -0.25);
        \draw[red, ->, >=latex] (0.5, -0.25) -- (1, -0.25);
        \draw[red, ->, >=latex] (-0.5, -0.25) -- (0, -0.25);
        \draw[red, <-, >=latex] (-1, -0.25) -- (-0.5, -0.25);
        \draw[red, <-, >=latex] (0, -0.25) -- (0.5, -0.25);
        \draw[red, <-, >=latex] (1, -0.25) -- (1.5, -0.25);
        \draw[knot] (1,0) to[out=90,in=-90] (0,1);
        \draw[knot] (0,0) to[out=90,in=-90] (1,1);
        \draw[knot] (1,1) to[out=90,in=-90] (2,2);
        \draw[knot] (1,2) to[out=-90,in=90] (2,1);
        \draw[knot] (-1,1) to[out=90,in=-90] (0,2);
        \draw[knot] (0,1) to[out=90,in=-90] (-1,2);

        \draw[ultra thick] (0,0) to (0,-1);
        \draw[ultra thick] (1,0) to (1,-1);
        \draw[ultra thick] (2,1) to (2,-1);
        \draw[ultra thick] (-1,1) to (-1,-1);

        \draw[red] (-0.5,-0.25) node[anchor = north]{$\alpha_1$};
        \draw[red] (0.5,-0.25) node[anchor = north]{$\alpha_2$};
        \draw[red] (1.5,-0.25) node[anchor = north]{$\alpha_3$};
        \filldraw[blue] (-1,-1) circle (2pt) node[anchor=north, yshift=-2pt]{$X_1$};
        \filldraw[blue] (0,-1) circle (2pt) node[anchor=north, yshift=-2pt]{$X_2$};
        \filldraw[blue] (1,-1) circle (2pt) node[anchor=north, yshift=-2pt]{$X_3$};
        \filldraw[blue] (2,-1) circle (2pt) node[anchor=north, yshift=-2pt]{$X_4$};
        \filldraw[] (-1,2) circle (2pt);
        \filldraw[] (0,2) circle (2pt);
        \filldraw[] (1,2) circle (2pt);
        \filldraw[] (2,2) circle (2pt);
    \end{tikzpicture}
    \caption{Khovanov-Rozansky homology has been historically defined in terms of the variables $X_i$ corresponding to the strands of the braid. Our construction defines it in terms of the variables $\alpha_i$ corresponding to the regions between the braids, which makes many formulas symmetric and overall simpler.}
    \label{fig:alphas_instead_of_xs}
\end{figure}
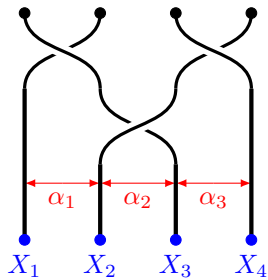

Theorem \ref{thm:equivalence Ras and Kho homology level} is based on a stronger result, the claim that suitably interpreted, Rasmussen's chain complex is a free resolution of the Rouquier chain complex. This is an original result and it is specific to both the reduced setting and to knots.
\begin{theorem}\label{thm:equivalence Ras and Kho}
Let $\beta \in \Br_n$ be a braid whose closure is a knot. Then $C_{\Ras}(\beta)$ is a free resolution of $C_\Rou(\beta)$ as chain complexes over $\F[\alpha_1, \dots, \alpha_{n-1}, \alpha_1', \dots, \alpha_{n-1}']$.
\end{theorem}
For example, Theorem \ref{thm:equivalence Ras and Kho} can be used to provide upper bounds on the projective dimension of Rouquier complexes or to show that Rasmussen's chain complex can be used to compute the Hochschild homology of a Rouquier complex.

\subsection*{Acknowledgment}
I would like to thank Sucharit Sarkar without whom I would have given up on understanding the intricacies of Khovanov-Rozansky homology long ago.

\section{Braid group, Hecke algebra, Markov trace}
Throughout the paper, let $\F$ denote an arbitrary field with $\mathrm{char}(\F) \neq 2$, except in Section \ref{sec:higher differentials}, where we require that $\mathrm{char}(\F) =0$.

\subsection{Braid group}
Let $n \in \N$ be a natural number. The braid group on $n$ strands is given by the presentation
$$\Br_n  = \left\langle \sigma_1, \dots, \sigma_{n-1} \middle|  \ 
\begin{array}{l}
\sigma_i \sigma_j = \sigma_j\sigma_i \text{ for all $i, j \in \{1, \dots, n-1\}$ with $|i-j|>1$} \\
\sigma_i\sigma_{i+1}\sigma_i = \sigma_{i+1}\sigma_i\sigma_{i+1} \text{ for all $i\in \{1, \dots, n-2\}$}
\end{array}
\right\rangle.$$
In presentations of related algebraic objects we encounter in this paper, we will generally omit explicit references to the indexing sets for relations, since they can be inferred from the context. For the rest of this paper, $\beta \in \Br_n$ is a braid on $n$ strands and $c \in \N$ is used to denote the number of crossings in $\beta$. 

\subsection{Hecke Algebra}
A braid analogue of the HOMFLY skein relation is the equation $\sigma_i -\sigma_i^{-1} = (q-q^{-1})$, in which $q$ denotes a formal variable. Due to the presence of $q$, such a requirement cannot be imposed in the braid group $\Br_n$ itself. Instead, let $T = \Z[q, q^{-1}]$ be the ring of Laurent polynomials in $q$ and let $T[\Br_n]$ be the associated group ring. The Hecke algebra is the quotient of $T[\Br_n]$ by the skein relation $\sigma_i - \sigma_i^{-1} = (q-q^{-1})$ or equivalently by $\sigma_i^2 = (q-q^{-1})\sigma_i + 1$. We can restate this algebraically as follows.
\begin{definition}
The \emph{Hecke algebra} is a $T$-algebra with the presentation
$$\H_n  = \left\langle \sigma_1, \dots, \sigma_{n-1} \middle|  \ 
\begin{array}{l}
\sigma_i \sigma_j = \sigma_j\sigma_i \text{ for all $i, j$ with $|i-j|>1$} \\
\sigma_i\sigma_{i+1}\sigma_i = \sigma_{i+1}\sigma_i\sigma_{i+1}\\
\sigma_i^2 = (q-q^{-1})\sigma_i + 1
\end{array}
\right\rangle.$$
\end{definition}
For $i \in \{1, \dots, n-1\}$ define $B_i = q^{-1} - \sigma_i$. It is easy to check that this gives another presentation of $\H_n$ in terms of the new generators as
$$\H_n  = \left\langle B_1, \dots, B_{n-1} \middle|  \ 
\begin{array}{l}
B_i B_j = B_j B_i \text{ for all $i, j$ with $|i-j|>1$} \\
B_i B_{i+1} B_i - B_i = B_{i+1} B_i B_{i+1} - B_{i+1} \\
B_i^2 = (q+q^{-1})B_i
\end{array}
\right\rangle.$$

\subsection{Markov trace}
By Markov's theorem, any two braids with isotopic closures are related by a sequence of conjugations and (de)stabilizations. Roughly speaking, a Markov trace is a family of maps $\mathrm{Tr}_n: \mathrm{H}_n \to \Z[a^{\pm 1}, q^{\pm 1}, \frac{1}{q-q^{-1}}]$ that behave well with respect with the two types of moves. Up to normalization, there exists a unique Markov trace on the Hecke algebras $\mathrm{H}_n$.
\begin{theorem}
\cite[attributed to Ocneanu]{Jones1987Hecke}
For each $n \in \N$, there exists a unique $T$-module homomorphism $\mathrm{Tr}_n: \mathrm{H}_n \to \Z[a^{\pm 1}, q^{\pm 1}, \frac{1}{q-q^{-1}}]$ such that
\begin{enumerate}
    \item $\mathrm{Tr}_n (BB') = \mathrm{Tr}_n (B'B)$ for all $B, B' \in \mathrm{H}_n$,
    \item $\mathrm{Tr}_{n+1} (i(B)) = \{0\} \mathrm{Tr}_{n}(B)$ and $\mathrm{Tr}_{n+1} (i(B)B_{n}) = \{1\} \mathrm{Tr}_{n}(B)$ for all $B \in \H_n$, and
    \item $\mathrm{Tr}_{1} (1) = 1$,
\end{enumerate}
where $i : H_n \to H_{n+1}$ is the inclusion map that sends $B_j \mapsto B_j$ for all $j \in \{1, \dots, n-1\}$ and $\{n\} = \frac{aq^{-n}-a^{-1}q^n}{q-q^{-1}}$.
\end{theorem}
The middle property might appear mysterious at first glance. However, recalling that $\sigma_n = q^{-1}-B_n$, one can show that the two equations reduce to $\mathrm{Tr}_{n+1}(i(B)\sigma_n) = a^{-1} \mathrm{Tr}_n(B)$. In other words, each Markov stabilization multiplies the trace by $a^{-1}$. Renormalizing back we obtain the definition of the HOMFLY polynomial in the language of braids.
\begin{theorem}
\cite{Jones1987Hecke}
Let $\beta \in \Br_n$ be a braid. Then the HOMFLY polynomial of its closure $\overline{\beta}$ is $P_{\, \overline{\beta}}(a, q) = a^{w} \mathrm{Tr}_n(\pi(\beta))$ where $w$ denotes the writhe of $\overline{\beta}$ and $\pi: T[\Br_n] \to \H_n$ is the canonical projection $\sigma_i \mapsto \sigma_i$.
\end{theorem}

\begin{example}\label{example:Trn}
Let $n = 2$ and consider $B_1 \in H_2$. Then
$$\mathrm{Tr}_2(B_1) = \{1\} = \frac{aq^{-1}-a^{-1}q}{q-q^{-1}}$$
using the properties $(2)$ and $(3)$ and 
\begin{align*}
\mathrm{Tr}_2(q^{-1}-B_1) &= q^{-1}\mathrm{Tr_2}(1) - \mathrm{Tr}_2(B_1) = q^{-1} \{0\} - \{1\} \\
&= q^{-1} \ \frac{a-a^{-1}}{q-q^{-1}} - \frac{aq^{-1}-a^{-1}q}{q-q^{-1}}\\
&= \frac{1}{q-q^{-1}} \left( q^{-1}a-q^{-1}a^{-1}-aq^{-1}+a^{-1}q \right) = a^{-1}
\end{align*}
using $T$-linearity and properties $(1)$, $(2)$, and $(3)$. Since $\overline{\sigma_1} = U$ is the unknot, it follows that $P_{U}(a, q) = a^1 a^{-1} = 1$ as required. Alternatively, one can check that $\mathrm{Tr}_2(q-B_1) = a$, which gives another computation of the HOMFLY polynomial for the unknot.
\end{example}
The next few sections are dedicated to categorifying $\H_n$, $\Br_n$, $\mathrm{Tr}_n$, and $P_K$ into Khovanov-Rozansky homology.

\section{Categorification of \texorpdfstring{$\mathrm{H}_n$}{Hn}: Soergel bimodules}
The Hecke algebra $\mathrm{H}_n$ was categorified by Soergel in \cite{soergel1990kategorie}. This subsection presents a variation of his construction.
\subsection{Representation of \texorpdfstring{$S_n$}{Sn}}
Let $\check{E}$ be an $n$-dimensional $\F$-vector space and let $\{e_1, \dots, e_n\}$ be a basis of $\check{E}$. There is a natural action of the symmetric group $S_n$ on $\check{E}$ via permuting the coordinates, \emph{i.e.} the permutation $\sigma \in S_n$ acts on the basis elements via $\sigma(e_i) = e_{\sigma(i)}$ for all $i$ and the action is extended linearly to all of $\check{E}$. Note that $e_1+\dots+e_n \in \check{E}$ spans an $S_n$-invariant subspace, so $\check{E}$ is not irreducible as a representation of $S_n$. In fact, $\check{E}$ decomposes into
$$\check{E} \cong \F\langle \alpha_1, \dots, \alpha_{n-1} \rangle \oplus \F\langle e_1+\dots+e_{n} \rangle$$
where $\alpha_i = e_{i+1}-e_i$ for all $i$. Note that both summands are irreducible.

\vspace{1em}

Let $E = \F\langle \alpha_1, \dots, \alpha_{n-1} \rangle$ be the standard representation. The action of $S_n$ on $E$ extends to the action of $S_n$ on the symmetric algebra $S(E) \cong \F[\alpha_1, \dots, \alpha_{n-1}]$. Since $S(E)$ is isomorphic to a polynomial ring, it is naturally $\N_0$-graded by degree, \emph{i.e.} there is a decomposition $S(E) \cong \bigoplus_{i \in \N_0} S(E)_{i}$ where $S(E)_{i}$ is the $\F$-vector space generated by homogeneous polynomials of degree $i$. For historic reasons, we almost exclusively deal with the double of this grading.
\begin{definition}
Define $\gr_q: \bigsqcup_{i \in \N_0} S(E)_i \to \N_0$ via $\gr_q(x) = 2i$ for all $x \in S(E)_i$.
\end{definition}
The grading $\gr_q$ is called the \emph{quantum grading}.

\subsection{Soergel bimodules}
Let $R = S(E)$ and $R^{\sigma_i}$ denote the subring of $R$ that is fixed under the action of the elementary transposition $\sigma_i = (i \ i+1)$. Note that $R$ is integral over $R^{\sigma_i}$.
\begin{remark}
For example, we have $\alpha_i \notin R^{\sigma_i}$ since $\sigma_i(\alpha_i) = -\alpha_i$, but $\alpha_i^2 \in R^{\sigma_i}$ since $\sigma_i(\alpha_i^2)=\sigma_i(\alpha_i)\sigma_i(\alpha_i)=(-\alpha_i)(-\alpha_i) = \alpha_i^2$.
\end{remark}
Since $\mathrm{char}(\F) \neq 2$, a good description of $R^{\sigma_i}$ is
$$R^{\sigma_i} = \F\left[\alpha_1, \dots, \alpha_{i-2}, \alpha_{i-1}+\frac{1}{2}\alpha_i, \alpha_i^2, \alpha_{i+1}+\frac{1}{2}\alpha_i, \alpha_{i+1}, \dots, \alpha_{n-1}\right].$$ 
This is because $\sigma_i(\alpha_j)=\alpha_j$ for all $i, j$ with $|i-j|>1$ and $\sigma_i(\alpha_{i-1}) = \alpha_{i-1} + \alpha_i$ and $\sigma_i(\alpha_{i+1}) = \alpha_{i+1}+\alpha_i$ and it follows that the variables of the above polynomial ring are all fixed under the action of $\sigma_i$.

\vspace{1em}

Consider $R$ as a module over $R^{\sigma_i}$. It is free of rank $2$ since $R \cong R^{\sigma_i} \langle 1, \alpha_i \rangle$. Note that this isomorphism is not canonical, \emph{i.e.} there are many choices of bases of $R$ as a free-module over $R^{\sigma_i}$. However, all homogeneous bases consist of elements $(x, y)$ satisfying $\gr_q(x) = 0$ and $\gr_q(y) = 2$. This allows us to write $R \cong (1+q^2) R^{\sigma_i}$ where multiplication by $q$ indicates a grading shift of $\gr_q$ upwards by $1$ for each element of the ring.

\vspace{1em}

The ring $R$ is a module over itself. It can also be treated as an $(R, R)$-bimodule with the same left and right actions, or more generally as an $(R', R'')$-bimodule for any subrings $R', R'' \subset R$. This observation leads to the definition of elementary Soergel bimodules.
\begin{definition}
An \emph{elementary Soergel bimodule} is $\B_i = q^{-1} R \otimes_{R^{\sigma_i}} R$.
\end{definition}
Note that $R$ and $\B_i$ are both $(R, R)$-bimodules. Every $(R, R)$-bimodule can also be treated as a module over $R \otimes_\F R \cong \F[\alpha_1, \dots, \alpha_{n-1}, \alpha_1', \dots, \alpha_{n-1}']$ and this interpretation is frequently easier to work with. We have the isomorphisms 
$$R \cong \frac{\F[\alpha_1, \dots, \alpha_{n-1}, \alpha_1', \dots, \alpha_{n-1}']}{(\alpha_i - \alpha_i' \text{ for all } i)}$$
and 
$$\B_i \cong q^{-1} \frac{\F[\alpha_1, \dots, \alpha_{n-1}, \alpha_1', \dots, \alpha_{n-1}']}{(\alpha_j-\alpha_j' \text{ for $j$ with $|i-j|>1$, } (\alpha_{i\pm 1}+\frac{1}{2}\alpha_i) - (\alpha_{i \pm 1}'+\frac{1}{2}\alpha_i'),  \alpha_i^2-\alpha_i'^2  )}.$$
\begin{definition}
A \emph{Soergel bimodule} is a direct summand of a tensor product of elementary Soergel bimodules.
\end{definition}
The category of Soergel bimodules $\mathrm{SBim}_n$ is a subcategory of $(R, R)\mathrm{-Bim}_{\Z}$. As such, it is additive and monoidal with the monoidal structure given by the tensor product $\otimes_{R \otimes R}$. We frequently abbreviate $MN = M \otimes_{R \otimes R} N$.
\begin{theorem}
\cite{soergel1990kategorie}
In the category of Soergel bimodules $\mathrm{SBim}_n$, we have the following isomorphisms:
\begin{enumerate}
    \item $\B_i\B_j \cong \B_j\B_i$ for all $i, j$ with $|i-j| > 1$,
    \item $\B_i\B_{i+1}\B_i \oplus \B_{i+1} \cong \B_{i+1}\B_i\B_{i+1} \oplus \B_i$, and
    \item $\B_i^2 \cong (q+q^{-1}) \B_i$.
\end{enumerate}
\label{thm:Soergel relations}
\end{theorem}
Since $\mathrm{SBim}_n$ is a graded additive monoidal category, its Grothendieck group $K_0(\mathrm{SBim}_n)$ can be defined as the free $T$-module generated by the objects of $\mathrm{SBim}_n$ modulo relations of the form $[\B_i] + [\B_j] = [\B_i\oplus \B_j]$ and $q[\B_i] = [q\B_i]$. The Grothendieck group becomes a $T$-algebra if we define multiplication via $[\B_i]\cdot[\B_j] = [\B_i\B_j]$.
\begin{theorem}
\cite{soergel1990kategorie}
There is an isomorphism of $T$-algebras $\Psi: \mathrm{H}_n \cong K_0(\mathrm{SBim}_n)$ given by $B_i \mapsto [\B_i]$. 
\end{theorem}

\section{Categorification of \texorpdfstring{$\Br_n$}{Brn}: Rouquier complexes}
Let $K(\mathrm{SBim}_n)$ denote the homotopy category of the category $\mathrm{SBim}_n$ of Soergel bimodules. In \cite{rouquier2006categorification}, a map $C_\Rou: \{ \sigma_1, \dots, \sigma_{n-1}, \sigma_1^{-1}, \dots, \sigma_{n-1}^{-1} \} \to K(\mathrm{SBim}_n)$ was defined via
$$C_\Rou(\sigma_i) = \B_i \xrightarrow{1} q^{-1}R$$
and
$$C_\Rou(\sigma_i^{-1}) = qR \xrightarrow{\alpha_i+\alpha_i'} \B_i$$
for all $i \in \{1, \dots, n-1\}$. The homological grading on $C_\Rou(\sigma_i)$ and $C_\Rou(\sigma_i^{-1})$ is called the \emph{Rouquier} grading and is denoted by $\gr_R$. It is chosen such that the terms $qR$ and $q^{-1}R$ are in Rouquier grading $0$ and the differential decreases $\gr_R$ by $1$.
\begin{definition}
Chain homotopy types $C_\Rou(\sigma_i)$ and $C_\Rou(\sigma_i^{-1})$ are \emph{elementary Rouquier complexes}. 
\end{definition}
Inheriting these properties from $\mathrm{SBim}_n$, the category $K(\mathrm{SBim}_n)$ is a monoidal category whose monoidal structure is given by the tensor product of chain complexes $\otimes_{R \otimes R}$. We likewise frequently omit the subscript here.
\begin{definition}
A \emph{Rouquier complex} is a tensor product of elementary Rouquier complexes. Let $\mathrm{Rou}_n \subset K(\mathrm{SBim}_n)$ denote the full subcategory of $K(\mathrm{SBim}_n)$ spanned by the Rouquier complexes. 
\end{definition}
\begin{theorem}
\label{thm:Rouquier relations}
\cite{rouquier2006categorification}
We have the following isomorphisms in $\mathrm{Rou}_n \subset K(\mathrm{SBim}_n)$.
\begin{enumerate}
    \item $C_\Rou(\sigma_i) \otimes C_\Rou(\sigma_i^{-1}) \cong R$,
    \item $C_\Rou(\sigma_i \sigma_j) \cong C_\Rou(\sigma_j \sigma_i)$ for all $i, j$ with $|i-j| > 1$, and
    \item $C_\Rou(\sigma_i\sigma_{i+1}\sigma_i) \cong C_\Rou(\sigma_{i+1}\sigma_{i}\sigma_{i+1})$.
\end{enumerate}
\end{theorem}
It is implicit in the statement of the theorem that the isomorphisms also preserve $\gr_q$. This follows immediately from the fact that $C_\Rou(\sigma_i^{\pm 1})$ preserve $\gr_q$.

\vspace{1em}

Based on the above relations, Rouquier proved in \cite{rouquier2006categorification} that the category $\mathrm{Rou}_n$ is a categorification of the braid group $\Br_n$.
\begin{theorem}
\cite{rouquier2006categorification}
There is a group isomorphism $\Br_n \cong (\mathrm{Rou}_n, \otimes_{R \otimes R})$.
\end{theorem}
Having categorified both $\Br_n$ and $\mathrm{H}_n$, we turn our attention to the canonical projection $\pi: T[\Br_n] \to \mathrm{H}_n$ defined by $\pi(\sigma_i) = \sigma_i = q^{-1} - B_i$. Since $\mathrm{SBim}_n$ is additive, there is a well-defined map $\chi: K(\mathrm{SBim}_n) \to K_0(\mathrm{SBim}_n)$ defined via $\chi(D) = \sum_{\gr_R} (-1)^{\gr_R} [D_{\gr_R}]$ where $D = \bigoplus_{\gr_R}D_{\gr_R}$ is a chain homotopy type. The fact that $\chi$ categorifies $\pi$, \emph{i.e.} that $\Psi^{-1} \circ \chi \circ C_\Rou = \pi$ can quickly be seen by considering how both sides act on $\sigma_i^{\pm 1}$.

\section{Categorification of \texorpdfstring{$\mathrm{Tr}_n$}{Trn}: Hochschild homology}
Hochschild homology is a functor
$$\mathrm{HH}: (R,R)\mathrm{-Bim} \to R\mathrm{-Mod}_\Z$$
that sends $(R,R)$-bimodules to $\Z$-graded $R$-modules. Let $M \in (R,R)\mathrm{-Bim}$. As we have already seen, $(R, R)$-bimodules can be treated as $R \otimes R$-modules. Hochschild homology of $M$ is defined as the left derived tensor product $$\mathrm{HH}(M) = M \otimes_{R\otimes R}^L R.$$
In practice, this is computed by starting with a free resolution $P_M$ of $M$ or a free resolution $P_R$ of $R$, tensoring them with the other $R \otimes R$-module and taking homology, \emph{i.e.} $\mathrm{HH}(M) = H(P_M \otimes_{R \otimes R} R) = H(M \otimes_{R \otimes R} P_R)$. If one wishes, both $M$ and $R$ can be resolved as well.

\vspace{1em}

Note that \emph{a priori} $\mathrm{HH}(M)$ is an $R\otimes R$-module as opposed to merely an $R$-module. However, the left and right actions on $R$ coincide, and hence so do the left and right actions on $\mathrm{HH}(M)$. As a result, no information is lost by treating $\mathrm{HH}(M)$ as an $R$-module.

\vspace{1em}

Finally, $\mathrm{HH}$ is a functor to the category of \emph{graded} $R$-modules. The $\Z$-grading $\gr_H$ is called the \emph{Hochschild grading} and comes from the homological grading on $P_M$. The normalization is such that if $\dots \to M_2 \to M_1 \to M_0 \to 0$ is a free resolution of $M$ by $R \otimes R$-modules, then any $m \in M_i$ satisfies $\gr_H(m) = i$. 

\vspace{1em}

In our case, the $(R,R)$-bimodules $M$ will be the elements of $\mathrm{SBim}_n$, \emph{i.e.} Soergel bimodules. In particular, these $M$ are equipped with an extra internal grading $\gr_q$. We require that the differential in $P_M$ increases $\gr_q$ by $2$. The $\mathrm{HH}(M)$ is thus equipped with $2$ gradings ($\gr_H$, $\gr_q$).
\begin{example}\label{example:HHB}
Consider $\B_1 \in \mathrm{SBim}_2$. We will calculate $\mathrm{HH}(\B_1)$ in two ways.
\begin{enumerate}
    \item Note that $\B_1 = q^{-1}\frac{\F[\alpha_1, \alpha_1']}{(\alpha_1^2-\alpha_1'^2)}$ has a free resolution
    $$P_{\B_1} = q\F[\alpha_1, \alpha_1'] \xrightarrow{\alpha_1^2-\alpha_1'^2} q^{-1}\F[\alpha_1, \alpha_1']$$
    and $R = \frac{\F[\alpha_1, \alpha_1']}{(\alpha_1-\alpha_1')}$. Therefore 
    \begin{align*}
    P_{\B_1} \otimes R &= q\frac{\F[\alpha_1, \alpha_1']}{(\alpha_1-\alpha_1')} \xrightarrow{0} q^{-1}\frac{\F[\alpha_1, \alpha_1']}{(\alpha_1-\alpha_1')}\\
    &= q\F[\alpha_1] \xrightarrow{0} q^{-1}\F[\alpha_1]
    \end{align*}
    It follows that $\mathrm{HH}(\B_1) = H^1q^1\F[\alpha_1] \oplus H^0q^{-1}\F[\alpha_1]$, where the powers of $H$ denote the Hochschild grading of the elements in the summand.
    \item Alternatively, we can resolve $R$ rather than $\B_1$. In this case, we use the Koszul free resolution of $R$
    $$P_R = \F[\alpha_1, \alpha_1'] \xrightarrow{\alpha_1 - \alpha_1'} \F[\alpha_1, \alpha_1']$$
    and tensor it with $\B_1$ to obtain
    $$\B_1 \otimes P_R = q^{-1}\frac{\F[\alpha_1, \alpha_1']}{(\alpha_1^2-\alpha_1'^2)} \xrightarrow{\alpha_1 - \alpha_1'} q^{-1} \frac{\F[\alpha_1, \alpha_1']}{(\alpha_1^2-\alpha_1'^2)}.$$
    Taking homology, we see that $\mathrm{HH}(\B_1) \cong H^{1}q^1\F[\alpha_1] \oplus H^0q^{-1}\F[\alpha_1]$ as above. Note the increase in $\gr_q$ since $q^{-1} \frac{(\alpha_1+\alpha_1')\F[\alpha_1, \alpha_1']}{(\alpha_1^2-\alpha_1'^2)} \cong q\frac{\F[\alpha_1, \alpha_1']}{(\alpha_1-\alpha_1')} \cong q\F[\alpha_1]$.
\end{enumerate}
\end{example}
The number of strands matters, as can be seen from the following example.
\begin{example}
    Note that $\mathrm{HH}(R) = H^0q^0\F$ for $R \in \mathrm{SBim}_1$ and $\mathrm{HH}(R) = H^{1}q^0 \F[\alpha_1] \oplus H^0q^0\F[\alpha_1]$ for $R \in \mathrm{SBim}_2$.
\end{example}
Hochschild homology extends to a functor $\mathrm{HH}$ on chain complexes $\mathrm{Ch}(\mathrm{SBim}_n)$ by applying it term-by-term to each Soergel bimodule and differential in a chain complex $C$. Since $C$ is additionally equipped with $\gr_R$ and $\mathrm{HH}$ commutes with direct sums, the resulting $\mathrm{HH}(C)$ will be equipped with $3$ gradings $(\gr_H, \gr_q, \gr_R)$. Finally, we note that chain homotopy equivalent chain complexes $C_1 \simeq C_2$ have chain homotopy equivalent free resolutions $P_{C_1} \simeq P_{C_2}$, which further implies that $P_{C_1} \otimes_{R \otimes R} R \simeq P_{C_2} \otimes_{R\otimes R} R$. Hochschild homology is by definition the homology of these complexes, so this argument shows that $\mathrm{HH}$ descends to a well-defined functor on $K(\mathrm{SBim}_n)$.

\vspace{1em}

In the context of this paper, the importance of Hochschild homology stems from the fact that it categorifies $\mathrm{Tr}_n$. Let $M = \bigoplus M_{\gr_H, \gr_q} \in R\mathrm{-Mod}_{\Z\oplus\Z}$ be a $\Z\oplus\Z$ graded $R$-module with respect to the grading $(\gr_H, \gr_q)$. Define the functor $\eta: R$-Mod$_{\Z\oplus\Z} \to \Z[a^{\pm 1}, q^{\pm 1}, \frac{1}{q-q^{-1}}]$ via 
$$\eta(M) = (-aq)^{n-1}\sum_{\gr_H} \sum_{\gr_q} (-1)^{-\gr_H} a^{-2\gr_H} q^{\gr_q} \dim_\F M_{\gr_H, \gr_q}.$$
A formal restatement of the fact that $\mathrm{HH}$ categorifies $\mathrm{Tr}_n$ is the following theorem by Khovanov.
\begin{theorem}
\cite{khovanov2007triply}
For any $B \in \mathrm{H}_n$, we have $\mathrm{Tr}_n(B) = \eta(\mathrm{HH}(\Psi(B)))$.
\end{theorem}

\begin{example}
    We have seen in Example \ref{example:Trn} that $\mathrm{Tr}_2(B_1) = \frac{aq^{-1}-a^{-1}q}{q-q^{-1}}$. We have also seen in Example \ref{example:HHB} that $\mathrm{HH}(\B_1) = H^1q^1\F[\alpha_1] \oplus H^0q^{-1}\F[\alpha_1]$. Since $n=2$ we have $\eta(\mathrm{HH}(\B_1)) = -aq (-\frac{a^{-2}q^1}{1-q^2} + \frac{a^0q^{-1}}{1-q^2}) = \frac{aq^{-1}-a^{-1}q}{q-q^{-1}}$ as required.
\end{example}
\begin{example}
    Consider $1 \in \mathrm{H}_2$. Note that $\mathrm{Tr}_2(1) = \{0\} = \frac{a-a^{-1}}{q-q^{-1}}$ by Example \ref{example:Trn} and correspondingly also $\eta(\mathrm{HH}(R)) = -aq(-\frac{a^{-2}}{1-q^2} + \frac{1}{1-q^2}) = \frac{a-a^{-1}}{q-q^{-1}}$ as required. 
\end{example}

\section{Khovanov-Rozansky homology}
We are finally in a position to make the following definition.
\begin{definition}
Let $K =\overline{\beta}$ be a knot, represented as the closure of some braid $\beta \in \mathrm{Br}_n$ for some $n \in \N$. The \emph{reduced Khovanov-Rozansky homology} of $K$ is $\widehat{\mathit{H}}(K) = H_*(\mathrm{HH}(C_\Rou(\beta)))$.  
\end{definition}
$\widehat{\mathit{H}}(K)$ is a $\Z\oplus\Z\oplus\Z$-graded vector space over $\F$. Instead of the gradings $(\gr_H, \gr_q, \gr_R)$ considered thus far, it is customary to pass to the new gradings $(\gr_a, \gr_q, \gr_t)$ via
\begin{align*}
\gr_a &= -2\gr_H + w + (n-1)\\
\gr_q &= \gr_q + (n-1)\\
\gr_t &= \gr_R - \gr_H + (n-1).
\end{align*}
Let us be very explicit about the unfortunate standard practice in the literature that $\gr_q$ appears both on the left and on the right, but with different meanings. The middle equation $\gr_q = \gr_q + (n-1)$ should be interpreted as saying that the \emph{old} $\gr_q$ that has been considered thus far should be shifted upwards by $n-1$ to obtain the \emph{new} $\gr_q$. Whenever we use the gradings $(\gr_a, \gr_q, \gr_t)$, it should always be assumed that the new $\gr_q$ is used.

\vspace{1em}

The main result about Khovanov-Rozansky homology is that it categorifies the HOMFLY polynomial.
\begin{theorem}
\cite{khovanov2008matrix}
The HOMFLY polynomial $P_K$ of a knot $K$ can be expressed as $P_K(a, q) = \sum_{\gr_a, \gr_q, \gr_t} (-1)^{\gr_t} a^{\gr_a} q^{\gr_q} \dim_\F \widehat{\mathit{H}}(K)_{\gr_a, \gr_q, \gr_t}$.
\label{thm:HOMFLY homology categorifies HOMFLY polynomial}
\end{theorem}
We calculate Khovanov-Rozansky homology for two diagrams of the unknot and for a negative trefoil.
\begin{example}
    Let $n = 2$ and consider $\sigma_1^{-1} \in \Br_2$ with closure $\overline{\sigma_1^{-1}} = U$. The writhe of the diagram is $w = -1$. The braid only has one crossing, so its Rouquier complex is just the elementary Rouquier complex $C_\Rou(\sigma_1^{-1}) = qR \xrightarrow{\alpha_1 + \alpha_1'} \B_1$ itself. Taking the Hochschild homology with the help of calculations in Example \ref{example:HHB} yields $\mathrm{HH}(C_\Rou(\sigma_1^{-1}))$ to be
    $$H^1q^1R^0\F[\alpha_1] \oplus H^0q^1R^0\F[\alpha_1] \xrightarrow{\begin{pmatrix}
    1 & 0\\
    0 & 2\alpha_1
    \end{pmatrix}} H^1q^1R^{-1}\F[\alpha_1] \oplus H^0q^{-1}R^{-1}\F[\alpha_1].$$
    Taking homology of the complex gives $\widehat{\mathit{H}}(U) \cong \F_{0, -1, -1}$ in terms of the gradings $(\gr_H, \gr_q, \gr_R)$. Translating to the gradings $(\gr_a, \gr_q, \gr_t)$ gives $\widehat{\mathit{H}}(U) \cong \F_{0, 0, 0}$. Note that this calculation is consistent with Theorem \ref{thm:HOMFLY homology categorifies HOMFLY polynomial} and Example \ref{example:Trn} since it decategorifies to $P_U = 1$.
\end{example}
\begin{example}
    Similarly, we also calculate the HOMFLY homology of the unknot with the use of a different diagram $\sigma_1$. We still have $n=2$, but $w = 1$. The Rouquier complex is $C_\Rou(\sigma_1) = \B_1 \xrightarrow{1} q^{-1}R$ and taking the Hochschild homology yields $\mathrm{HH}(C_\Rou(\sigma_1^{-1}))$ to be  
    $$H^1q^1R^1\F[\alpha_1] \oplus H^0q^{-1}R^1\F[\alpha_1] \xrightarrow{\begin{pmatrix}
    2\alpha_1 & 0\\
    0 & 1
    \end{pmatrix}} H^1q^{-1}R^0\F[\alpha_1] \oplus H^0q^{-1}R^0\F[\alpha_1].$$
    Taking homology of the complex gives $\widehat{\mathit{H}}(U) \cong \F_{1, -1, 0}$ in terms of the gradings $(\gr_H, \gr_q, \gr_R)$. Translating to the gradings $(\gr_a, \gr_q, \gr_t)$ gives $\widehat{\mathit{H}}(U) \cong \F_{0, 0, 0}$. This coincides with the example above.
\end{example}
\begin{example}
    Let $n = 2$ and consider $\sigma_1^{-3} \in \Br_2$ whose closure is the negative trefoil $\overline{\sigma_1^{-3}} = T_{2, -3}$. We have $w = -3$. The Rouquier complex can be seen to be homotopy equivalent to 
    $$C_\Rou(\sigma_1^{-3}) = \otimes_{i=1}^3 C_\Rou(\sigma_1^{-1}) \simeq (q^3 R \xrightarrow{\alpha_1+\alpha_1'} q^2 \B_1 \xrightarrow{0} \B_1 \xrightarrow{\alpha_1+\alpha_1'} q^{-2}\B_1).$$
    Taking the Hochschild homology yields $\mathrm{HH}(C_\Rou(\sigma_1^{-3}))$ to be 
$$
\begin{tikzcd}
    H^1q^{-1}R^{-3}\F[\alpha_1] & H^0q^{-3}R^{-3}\F[\alpha_1]\\
    H^1q^1R^{-2}\F[\alpha_1] \arrow[u, "2\alpha_1"] & H^0q^{-1}R^{-2}\F[\alpha_1] \arrow[u, "2\alpha_1", swap]\\
    H^1q^3R^{-1}\F[\alpha_1] & H^0q^1R^{-1}\F[\alpha_1]\\
    H^1q^3R^0\F[\alpha_1] \arrow[u, "1"] & H^0q^3R^{0}\F[\alpha_1] \arrow[u, "2\alpha_1", swap]
\end{tikzcd}
$$
so $\widehat{\mathit{H}}(T_{2,-3}) \cong \F_{1, -1, -3} \oplus \F_{0, -3, -3} \oplus \F_{0, 1, -1}$ in terms of the gradings $(\gr_H, \gr_q, \gr_R)$. Translating the result to gradings $(\gr_a, \gr_q, \gr_t)$ yields $\widehat{\mathit{H}}(T_{2,-3}) \cong \F_{-4, 0, -3} \oplus \F_{-2, -2, -2} \oplus \F_{-2, 2, 0}$. This result is expected since the HOMFLY polynomial of the negative trefoil is $P_{T_{2, -3}} = -a^{-4} + a^{-2}q^{-2} + a^{-2}q^2$, which is in agreement with Theorem \ref{thm:HOMFLY homology categorifies HOMFLY polynomial}.
\end{example}
\begin{remark}
    Introducing new gradings at the final step might seem rather artificial. Indeed, the gradings $(\gr_a, \gr_q, \gr_t)$ could also be defined directly by specifying that the Hochschild differential increases $\gr_t$ by $1$, $\gr_a$ by $2$, and some normalization conditions for pinpointing the grading shifts.
\end{remark}

\section{Rasmussen's construction}
While the described categorification of the HOMFLY polynomial is algebraically slick and relatively amenable to calculations, the non-canonical choice of a free resolution for Hochschild homology in practice makes it challenging to establish structural properties of Khovanov-Rozansky homology. This section introduces a different approach to Khovanov-Rozansky homology due Rasmussen \cite{rasmussen2016some} with a complementary character -- it is difficult to calculate with, but very useful for proving theoretical results. We describe a variation of his construction that has been adapted for reduced Khovanov-Rozansky homology. In our construction, the variables of the polynomial ring correspond to the regions of the braid rather than to its strands.

\vspace{1em}

Let $\beta \in \Br_n$ be a braid with $n$ strands. In the standard fashion, we can express $\beta$ as a product of elementary braids $\beta = \prod_{k = 1}^c \sigma_{i_k}^{a_k}$ where $i_k \in \{1, \dots, n-1\}$ and $a_k \in \{1, -1\}$. As was previously the case with Rouquier complexes, the construction of Rasmussen's double chain complex $C_{\Ras}(\beta)$ is local in the sense that it assigns double chain complexes to individual crossings and tensors them together to obtain an invariant of the entire braid.

\subsection{Ring \texorpdfstring{$Q_\beta$}{Qbeta}}
We first define $Q_\beta$, the underlying ring of the Rasmussen's double chain complex $C_{\Ras}(\beta)$. Let $1, \dots, n-1+c$ denote the planar regions of $\beta$, labeled from left to right and bottom to top. See Figure \ref{fig:regions of sigmai} for an example where the braid is a single crossing $\sigma_i$. For each $i$, we associate a variable $\alpha_i$ to the bottom of the region $i$.
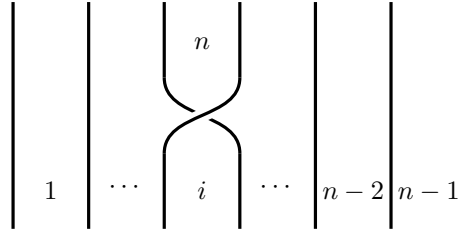
\begin{figure}[t]
    \begin{tikzpicture}[scale=1]
    \draw[knot] (1,0) to[out=90,in=-90] (0,1);
    \draw[knot] (0,0) to[out=90,in=-90] (1,1);
    \draw[very thick] (1,-1) to (1,0);
    \draw[very thick] (1,1) to (1,2);
    \draw[very thick] (0,-1) to (0,0);
    \draw[very thick] (0,1) to (0,2);
    \draw[very thick] (2,-1) to (2,2);
    \draw[very thick] (3,-1) to (3,2);
    \draw[very thick] (-1,-1) to (-1,2);
    \draw[very thick] (-2,-1) to (-2,2);
    \draw[very thick] (4,-1) to (4,2);

    \draw[] (-1.5,-0.25) node[anchor = north]{$1$};
    \draw[] (-0.5,-0.25) node[anchor = north]{$\cdots$};
    \draw[] (0.5,-0.25) node[anchor = north]{$i$};
    \draw[] (1.5,-0.25) node[anchor = north]{$\cdots$};
    \draw[] (2.5,-0.25) node[anchor = north]{$n-2$};
    \draw[] (3.5,-0.25) node[anchor = north]{$n-1$};
    \draw[] (0.5,1.25) node[anchor = south]{$n$};
    \end{tikzpicture}  
    \caption{Regions associated to an elementary braid $\sigma_i$.}
    \label{fig:regions of sigmai}
\end{figure}

\begin{definition}\label{def:Rasmussen's ring}
$Q_\beta = \F[\alpha_1, \dots, \alpha_{n-1+c}]$.
\end{definition}
We now define \emph{bottom} and \emph{top} actions on $Q_\beta$. The bottom action is simply multiplication by the variables $\alpha_1, \dots, \alpha_{n-1}$. The top action is multiplication by the variables $\alpha_1', \dots, \alpha_{n-1}'$ which are certain linear combinations of the original variables $\alpha_1, \dots, \alpha_{n-1+c}$. We define them inductively, beginning with a single crossing $\sigma_i^{\pm 1}$ with $Q_{\sigma_i^{\pm 1}} = \F[\alpha_1, \dots, \alpha_n]$. The variables $\alpha_1', \dots, \alpha_{n-1}'$ for the top action are
\begin{align*}
    &\alpha_i' = \alpha_n,\\
    &\alpha_{i \pm 1}' = \alpha_{i\pm 1} + \frac{1}{2}(\alpha_i-\alpha_n), \text{ and }\\
    &\alpha_{j}' = \alpha_j \text{ for $j$ with $|i-j|>1$.} 
\end{align*}
The top action for a general braid $\beta$ is defined inductively through the following result.
\begin{proposition}\label{prop:braid rings}
Let $\beta_1, \beta_2 \in \Br_n$ be braids. Then $Q_{\beta_1\beta_2} \cong Q_{\beta_1} \otimes_{\F[\gamma_1, \dots, \gamma_{n-1}]} Q_{\beta_2}$ where $\F[\gamma_1, \dots, \gamma_{n-1}]$ acts on $Q_{\beta_1}$ by the top and on $Q_{\beta_2}$ by the bottom action.
\end{proposition}
This allows one to use the top action on $Q_{\beta_2}$ to define the top action on $Q_{\beta_1\beta_2}$. Since the top action has already been defined on individual crossings, Proposition \ref{prop:braid rings} extends it to an arbitrary braid $\beta$.
\begin{proof}[Proof of Proposition \ref{prop:braid rings}]
Let the braids $\beta_1, \beta_2$ have $c_1, c_2$ crossings respectively. Tensor product identifies the bottom action on $Q_{\beta_2}$ with the top action on $Q_{\beta_1}$. The bottom action on $Q_{\beta_2}$ is given by multiplication by $\alpha_1, \dots, \alpha_{n-1} \in Q_{\beta_2}$ and the top action on $Q_{\beta_1}$ is given by multiplication by $\alpha_1', \dots, \alpha_{n-1}' \in Q_{\beta_1}$, which are linear combinations of $\alpha_1, \dots, \alpha_{n-1+c_1} \in Q_{\beta_1}$. It follows that $Q_{\beta_1} \otimes_{\F[\gamma_1, \dots, \gamma_{n-1}]} Q_{\beta_2}$ is a polynomial ring with variables $\alpha_1, \dots, \alpha_{n-1+c_1} \in Q_{\beta_1}$ and $\alpha_{n}, \dots, \alpha_{n-1+c_2} \in Q_{\beta_2}$. In other words, tensor product omits the variables $\alpha_1, \dots, \alpha_{n-1} \in Q_{\beta_2}$, since they are linear combinations of the variables from $Q_{\beta_1}$. In our pictorial definition, this means that $Q_{\beta_1} \otimes_{\F[\gamma_1, \dots, \gamma_{n-1}]} Q_{\beta_2}$ is generated by the regions of $\beta_1$ and $\beta_2$ except the bottom regions of $\beta_2$. This is what we want, since in $\beta_1\beta_2$, the bottom regions of $\beta_2$ become identified with the top regions of $\beta_1$.
\end{proof}
\begin{example}
Let us compute the top action for $\beta = \sigma_1\sigma_2 \in \Br_3$. Let $Q_{\sigma_1} = \F[\alpha_1, \alpha_2, \alpha_3]$ and $Q_{\sigma_2} = \F[\alpha_4, \alpha_5, \alpha_6]$ where the regions are labeled as shown in Figure \ref{fig:sigma1sigma2}. The top action on $Q_{\sigma_1}$ is given by $\alpha_1' = \alpha_3$ and $\alpha_2' = \alpha_2 + \frac{1}{2}(\alpha_1-\alpha_3)$ and the top action on $Q_{\sigma_2}$ is given by $\alpha_4' = \alpha_4+\frac{1}{2}(\alpha_5-\alpha_6)$ and $\alpha_5' = \alpha_6$. Identifying the top action on $Q_{\sigma_1}$ with the bottom action on $Q_{\sigma_2}$ as in the proof of Proposition \ref{prop:braid rings} gives that $Q_{\sigma_1\sigma_2} \cong \frac{\F[\alpha_1, \dots, \alpha_6]}{(\alpha_4 = \alpha_3, \alpha_5 = \alpha_2 + \frac{1}{2}(\alpha_1-\alpha_3))} \cong \F[\alpha_1, \alpha_2, \alpha_3, \alpha_6]$. Note that these variables correspond precisely to the regions of the composite braid $\sigma_1\sigma_2$. The top action on $Q_{\sigma_1\sigma_2}$ is now given by $\alpha_4' = \alpha_4 + \frac{1}{2}(\alpha_5-\alpha_6) = \frac{1}{4}\alpha_1+\frac{1}{2}\alpha_2+\frac{3}{4}\alpha_3-\frac{1}{2}\alpha_6$ and $\alpha_5' = \alpha_6$.
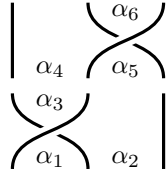
\begin{figure}[t]
    \begin{tikzpicture}[scale=1.0]
    \def\a{0.2}
    \draw[knot] (1,0) to[out=90,in=-90] (0,1);
    \draw[knot] (0,0) to[out=90,in=-90] (1,1);
    \draw[knot] (2,1+\a) to[out=90,in=-90] (1,2+\a);
    \draw[knot] (1,1+\a) to[out=90,in=-90] (2,2+\a);
    \draw[very thick] (2,0) to (2,1);
    \draw[very thick] (0,1+\a) to (0,2+\a);
    \draw[] (0.5, +0.35) node[anchor = north]{$\alpha_1$};
    \draw[] (1.5, +0.35) node[anchor = north]{$\alpha_2$};
    \draw[] (0.5, +0.65) node[anchor = south]{$\alpha_3$};
    \draw[] (0.5, +1.35+\a) node[anchor = north]{$\alpha_4$};
    \draw[] (1.5, +1.35+\a) node[anchor = north]{$\alpha_5$};
    \draw[] (1.5, +1.65+\a) node[anchor = south]{$\alpha_6$};
    \end{tikzpicture}  
    \caption{$\beta = \sigma_1 \sigma_2$. The calculation of the top action on $\beta$ is given in Example \ref{example:sigma1sigma2}.}
    \label{fig:sigma1sigma2}
\end{figure}
\end{example}\label{example:sigma1sigma2}

\subsection{Dual variables and flows}\label{subsection:Flows}
This subsection contains some technical tools that will be useful in the proof of Theorem \ref{thm:equivalence Ras and Kho}, but do not appear elsewhere in the paper. Therefore, we recommend that it be skipped during the first reading of the paper. 

Recall that each region $i \in \{1, \dots, n-1+c\}$ of $\beta$ received a corresponding variable $\alpha_i$ associated to its bottom part. We shall now introduce \emph{dual} variables $\alpha_1^\vee, \dots, \alpha_{n-1+c}^\vee \in Q_\beta$ that represent the top parts of the same regions. Intuitively speaking, ``we turn the braid upside down'' and repeat the construction of the previous subsection. More formally, the dual variables $\alpha_1^\vee, \dots, \alpha_{n-1+c}^\vee$ are certain linear combinations of $\alpha_1, \dots, \alpha_{n-1+c}$ and are defined in the following manner.
\begin{itemize}
    \item If $i$ is the open region lying near the top of $\beta$ between the $j^{\text{th}}$ and $(j+1)^{\text{st}}$ strand, then $\alpha_i^\vee = \alpha_j'$.
    \item If $i$ is a closed region, then $\beta = \beta_1 \sigma_j^{\pm 1} \beta_2$ where $j$ is such that the highest part of $i$ lies between the $j^{\text{th}}$ and $(j+1)^{\text{st}}$ strand and $\beta_1$ and $\beta_2$ are the parts of $\beta$ below and above this crossing respectively. Then there is a well-defined action of $\alpha_j$ on $Q_{\sigma_j^{\pm 1}}$ and therefore on $Q_\beta = Q_{\beta_1 \sigma_j^{\pm 1} \beta_2}$ by Proposition \ref{prop:braid rings}. Define $\alpha_i^\vee$ to be this action.
\end{itemize}
The above definition is rather theoretical and does not explain how to express $\alpha_i^\vee$ as a linear combination of $\alpha_1, \dots, \alpha_{n-1+c}$ in practice. The remaining part of this subsection describes a geometric way of resolving this problem.

\vspace{1em}

Draw a braid $\beta$ and mark the variables $\alpha_1, \dots, \alpha_{n-1+c}$ by black dots near the bottom of their regions. Additionally, mark the variable $\alpha_i^\vee$ by a red dot near the top of the region $i$. We interpret black dots as \emph{sinks}, the red dot as a \emph{source} we find an expression of $\alpha_i^\vee$ as a linear combination of $\alpha_1, \dots, \alpha_{n-1+c}$ by \emph{flowing} from the red dot to the black ones, following some rules.
\begin{enumerate}
    \item In each region, we flow in a vertical direction towards the black dot.
    \item Each part of the flow has an associated coefficient, which we call the \emph{flow rate}. The initial flow rate near the red dot is $1$. If the strand immediately to the right or immediately to the left of the flow with rate $c$ enters a crossing, a part of the flow splits away and enters the parallel regions with flow rates $\pm \frac{c}{2}$. If the flow changes a direction, the rate is $-\frac{c}{2}$ and if the flow continues flowing in the same direction, the rate is $\frac{c}{2}$.
    \item We flow until all parts of the flow have reached the black dots.
\end{enumerate}
Parts of the flow are called \emph{flowlines}. Let $c_j$ be the \emph{total flow rate} near the black dot associated to $\alpha_j$, \emph{i.e.} the sum of flow rates of all flowlines that have reached this point. Then $\alpha_i^\vee = \sum_{j=1}^{n-1+c}c_j\alpha_j$.

\begin{example}
Let us demonstrate the flowing process on $\beta = \sigma_1 \sigma_2 \sigma_1^{-1} \in \Br_3$ to express $\alpha_3^{\vee}$ as a linear combination of $\alpha_1, \alpha_2, \alpha_3, \alpha_4$, and $\alpha_5$. The braid $\beta$ is depicted in Figure \ref{fig:example of flow} together with $5$ black dots representing the sinks $\alpha_1, \alpha_2, \alpha_3, \alpha_4$, and $\alpha_5$ and $1$ red dot representing the source $\alpha_3^\vee$. The flow starts at the red dot and flows downwards towards the black dot $\alpha_3$. Its flow rate is $1$. At some point, the strand immediately to the left of the flow enters the crossing $\sigma_2$ and we are in the situation from (2). The flow splits into three flowlines with coefficients $1$, $\frac{1}{2}$ and $-\frac{1}{2}$. More precisely, the flowline that continues straight downwards has rate $1$, the flowline that changes the direction and flows upwards has rate $-\frac{1}{2}$ and the remaining flowline has rate $\frac{1}{2}$. Each of these three flowlines flows onward, obeying the same rules and undergoing more splits, until they eventually reach the black dots. Figure \ref{fig:example of flow} depicts the final stage of the flow. Let us now express $\alpha_3^{\vee}$ as a linear combination of $\alpha_1, \dots, \alpha_5$. Note that there are a total of two ways of reaching $\alpha_3$, one with coefficient $1$ and another with coefficient $-\frac{1}{4}$. Therefore, $c_3 = 1-\frac{1}{4} = \frac{3}{4}$. All other black dots can be reached at most once. Keeping track of the corresponding flow rates gives us $\alpha_3^{\vee} = \frac{1}{4}\alpha_1 + \frac{1}{2}\alpha_2 + \frac{3}{4}\alpha_3 - \frac{1}{2}\alpha_4$. 
\end{example}
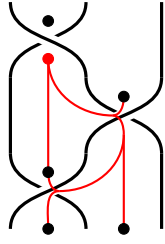
\begin{figure}[t]
    \begin{tikzpicture}[scale=1]
    \draw[knot] (1,0) to[out=90,in=-90] (0,1);
    \draw[knot] (0,0) to[out=90,in=-90] (1,1);
    \draw[knot] (2,1) to[out=90,in=-90] (1,2);
    \draw[knot] (1,1) to[out=90,in=-90] (2,2);
    \draw[knot] (0,2) to[out=90,in=-90] (1,3);
    \draw[knot] (1,2) to[out=90,in=-90] (0,3);
    
    \draw[red, thick] (0.5,2.25) to[out=-90,in=180] (1.35,1.5);
    \draw[red, thick] (1.35, 1.5) to[out=0,in=270] (1.5,1.75);
    \draw[red, thick] (1.35, 1.5) to[out=0,in=90] (1.5,1.25);
    \draw[red, thick] (0.5,2.25) -- (0.5,0.75);
    \draw[red, thick] (1.5,1.25) -- (1.5,0);
    \draw[red, thick] (1.5,1.25) to[out=270,in=0] (0.65,0.5);
    \draw[red, thick] (0.65, 0.5) to[out=180,in=270] (0.5,0.75);
    \draw[red, thick] (0.65, 0.5) to[out=180,in=90] (0.5,0);
    
    \draw[very thick] (2,2) to (2,3);
    \draw[very thick] (0,1) to (0,2);
    \draw[very thick] (2,0) to (2,1);

    \filldraw[] (0.5,0) circle (2pt);
    \filldraw[] (1.5,0) circle (2pt);
    \filldraw[] (0.5,0.75) circle (2pt);
    \filldraw[] (0.5,2.75) circle (2pt);
    \filldraw[] (1.5,1.75) circle (2pt);
    \filldraw[red] (0.5,2.25) circle (2pt);
    \end{tikzpicture}  
    \caption{The red dot denotes the dual variable $\alpha_3^{\vee}$ and the black dots denote the original variables $\alpha_1, \dots, \alpha_5$ ordered from left to right bottom to top. Expressing $\alpha_3^{\vee}$ as a linear combination of $\alpha_1, \dots, \alpha_5$ can be done by flowing from the red dot to the black dots.}
    \label{fig:example of flow}
\end{figure}

Recall that one of the rules governing the flowing process was that we flow until all flowlines have reached the black dots. This was perhaps slightly imprecise, because we have neglected to consider the possibility that some flowline enters a loop and never reaches a black dot. This can and indeed does happen in general. To rigorously resolve the issue, we must instead consider the limit of the flow as $t \to \infty$.
\begin{lemma}\label{lemma:flowing}
The flow from $\alpha_i^\vee$ to $\alpha_1, \dots, \alpha_{n-1+c}$ converges as $t \to \infty$. If $\alpha_i^\vee = \sum_{j=1}^{n-1+c}c_j\alpha_j$ for some unique $c_1, \dots, c_{n-1+c} \in \R$, then $|c_j| \leq 1$ for all $j$.
\end{lemma}
\begin{proof}
We start the flow from the red dot $\alpha_i^{\vee}$ at time $t=0$. In what follows, we prove that at all times $t \in [0, \infty)$ and all locations in $\beta$, the absolute total flow rate is $\leq 1$. We first observe that this is the case at $t=0$, since the flow rate near $\alpha_i^\vee$ is $1$ and the flow rate everywhere else is $0$. We also note that due to the way flow rates were defined in (2), all flowlines that flow downwards have a positive flow rate and all flowlines that flow upwards have a negative flow rate. Moreover, flowlines in all regions of the braid $\beta$ are flowing towards the black dot, \emph{i.e.} at every location in $\beta$, all flowlines are flowing in the same direction. This means that they all have a positive or they all have a negative sign. Stated differently, as the flowing process is ongoing, the absolute value of the total flow rate at every location of $\beta$ is increasing. Our next step is to show that it is also bounded above by $1$. Let $d$ be the vertical distance between two consecutive crossings and let all flowlines have the same constant speed that enables them to travel the vertical distance $d$ in time $\Delta t=1$. With this normalization convention, observe that the set of times when the flowlines are splitting can be identified with a subset $S \subset \frac{1}{2}+\N_0 \subset [0, \infty)$. The rest of the time, not much interesting is happening and it is therefore sufficient for us to discretize the time and analyze the flowline behavior only at $\N_0 \subset [0, \infty)$. We have already noted that the claim is true at $t=0$. In the rest of the argument, we perform an induction on $t$. Suppose that the absolute total flow rate at every location of $\beta$ at time $t=\tau$ is $\leq 1$. Let $t = \tau + 1$ and choose a location $\times$ in some region of $\beta$. Without loss of generality, let $\times$ be above the black dot in this region, for example as depicted in Figure \ref{fig:region <= 1 proof}.
\begin{figure}[t]
    \begin{subfigure}[b]{0.43\textwidth}
        \centering
    \begin{tikzpicture}[scale=1]
    \draw[knot] (1,0) to[out=90,in=-90] (0,1);
    \draw[knot] (0,0) to[out=90,in=-90] (1,1);
    \draw[knot] (0,3) to[out=90,in=-90] (1,4);
    \draw[knot] (1,3) to[out=90,in=-90] (0,4);
    
    \draw[very thick] (1,1) to (1,3);
    \draw[very thick] (0,1) to (0,3);

    \draw[] (0.5, 2) node[]{$\times$};
    \draw[] (0.5, 3) node[]{$\times_\tau$};

    \filldraw[] (0.5,0.75) circle (2pt);
    \end{tikzpicture}
        \caption{All flowlines that reach $\times$ at time $t=\tau+1$ must have already reached $\times_\tau$ at time $t=\tau$. This allows us to bound the total flow at $\times$ at time $t=\tau+1$.}
        \label{fig:region <= 1 proof a}
    \end{subfigure}
    \hspace{1cm}
    \begin{subfigure}[b]{0.43\textwidth}
    \centering
    \begin{tikzpicture}[scale=1]
    \draw[knot] (1,0) to[out=90,in=-90] (0,1);
    \draw[knot] (0,0) to[out=90,in=-90] (1,1);
    \draw[knot] (0,3) to[out=90,in=-90] (1,4);
    \draw[knot] (1,3) to[out=90,in=-90] (0,4);
    
    \draw[very thick] (1,1) to (1,3);
    \draw[very thick] (0,1) to (0,3);

    \draw[] (0.5, 3) node[]{$\times$};
    \draw[] (-0.5, 3.5) node[]{$\times_{\tau, 1}$};
    \draw[] (1.5, 3.5) node[]{$\times_{\tau, 2}$};

    \filldraw[] (0.5,0.75) circle (2pt);
    \end{tikzpicture}
        \caption{All flowlines that reach $\times$ at time $t=\tau+1$ must have already reached $\times_{\tau, 1}$ or $\times_{\tau, 2}$ at time $t=\tau$. This allows us to bound the total flow at $\times$ at time $t=\tau+1$.}
        \label{fig:region <= 1 proof b}
    \end{subfigure}
    \caption{} 
    \label{fig:region <= 1 proof}
\end{figure}

Since all flowlines at $\times$ are flowing downwards at a known speed, one can infer the state of the flow at time $\tau$. There are two possible qualitatively different states, depending on the proximity of $\times$ to the top of the region. If $\times$ is far from the top of the region, as in Figure \ref{fig:region <= 1 proof a}, then the total flow rate at $\times$ at time $\tau+1$ is the same as the total flow rate at $\times_\tau$ at time $\tau$. This is because every flowline that reaches $\times$ at time $\tau+1$ must have already reached $\times_\tau$ at time $\tau$. The inductive hypothesis states that the absolute total flow rate at $\times_\tau$ at time $\tau$ is $\leq 1$, which proves the inductive step. The other case is similar. If $\times$ is near the top of the region, as in Figure \ref{fig:region <= 1 proof b}, then all flowlines at $\times$ at time $\tau+1$ must have come from the locations $\times_{\tau, 1}$ and $\times_{\tau, 2}$ at time $\tau$. Both of their absolute flow rates are bounded by $1$ by the inductive hypothesis and crossing into a new region comes with a reduction of the flow rate by a factor of $2$. Therefore, the total flow rate at $\times$ at time $\tau+1$ is at most $\frac{1}{2} \cdot 1 + \frac{1}{2} \cdot 1 = 1$, as required.

\vspace{1em}

We have shown above that at all times $t \in [0, \infty)$ and all locations in $\beta$, the absolute total flow rate is $\leq 1$. We have also noted that the absolute total flow rate is increasing everywhere as time goes on. By completeness of $\R$, the process converges as we take the limit $t \to \infty$. The constants $c_j$ are the limits of total flow rates near the black dots, proving the lemma.
\end{proof}

There are several ways in which the flowing process can  be generalized slightly. Firstly, one can \emph{pause} the flow before all flowlines reach the black dots. In this case, one still obtains a valid expression of $\alpha_i^\vee$ as a linear combination of the variables associated to the endpoints of the flowlines. However, the endpoints of the flowlines are not necessarily linearly independent in this case, so the expression is not unique in general. In a similar vein, there is nothing special about the basis $\{\alpha_1, \dots, \alpha_{n-1+c}\}$ that corresponds to the black dots near the bottom of the regions. Other placements of black dots work just as well, provided that they form a basis of $\F\langle \alpha_1, \dots, \alpha_{n-1+c} \rangle$. 

\begin{lemma}\label{lemma:dual variables as linear combinations}
    Let $\{\alpha_1^\circ, \dots, \alpha_{n-1+c}^\circ\}$ be a basis of  $\F\langle \alpha_1, \dots, \alpha_{n-1+c} \rangle$ where $\alpha_j^\circ \in \{\alpha_j, \alpha_j^\vee\}$ for each $j$. Let $\alpha_i^\vee = \sum_{j=1}^{n-1+c}c_j\alpha_j^\circ$ for some unique $c_1, \dots, c_{n-1+c} \in \R$. Then $|c_j| \leq 1$ for all $j$.
\end{lemma}
\begin{proof}
If $\alpha_i^{\circ} = \alpha_i^\vee$, then $c_i = 1$ and $c_j=0$ for all $j \neq i$, as required. The interesting case is when $\alpha_i^{\circ} = \alpha_i$. Its proof is identical to the proof of Lemma \ref{lemma:flowing}
\end{proof}

\subsection{Double chain complex \texorpdfstring{$C_{\Ras}(\beta)$}{CQbeta}}
Consider a single negative crossing $\sigma_i^{-1}$ and recall that $Q_{\sigma_i^{-1}} = \F[\alpha_1, \dots, \alpha_{n-1}, \alpha_n]$. Let $C_{\Ras}(\sigma_i^{-1})$ be the double chain complex
$$
\begin{tikzcd}
    Q_{\sigma_i^{-1}} \arrow[r, "\alpha_i^2-\alpha_i'^2"] & Q_{\sigma_i^{-1}}\\
    Q_{\sigma_i^{-1}} \arrow[u, "1"] \arrow[r, "\alpha_i-\alpha_i'"] & Q_{\sigma_i^{-1}} \arrow[u, "\alpha_i+\alpha_i'", swap]
\end{tikzcd}
$$
over $Q_{\sigma_i^{-1}}$. Let $d_+$ denote the horizontal differential and $d$ the vertical differential. We adopt similar conventions for positive crossings and define $C_{\Ras}(\sigma_i)$ to be the double chain complex
$$
\begin{tikzcd}
    Q_{\sigma_i} \arrow[r, "\alpha_i-\alpha_i'"] & Q_{\sigma_i}\\
    Q_{\sigma_i} \arrow[u, "\alpha_i+\alpha_i'"] \arrow[r, "\alpha_i^2-\alpha_i'^2"] & Q_{\sigma_i} \arrow[u, "1", swap]
\end{tikzcd}
$$
over $Q_{\sigma_i}$. Having defined the invariant for crossings, we now inductively extend it to all braids. Let $\beta_1, \beta_2 \in \Br_n$ be two braids with double chain complexes $C_{\Ras}(\beta_1)$ and $C_{\Ras}(\beta_2)$ over rings $Q_{\beta_1}$ and $Q_{\beta_2}$ respectively.
\begin{definition}\label{def:tensor product of Rasmussen's complexes}
$C_{\Ras}(\beta_1\beta_2) = (C_{\Ras}(\beta_1) \otimes_{Q_{\beta_1}} Q_{\beta_1\beta_2}) \otimes_{Q_{\beta_1\beta_2}} (C_{\Ras}(\beta_2) \otimes_{Q_{\beta_2}} Q_{\beta_1\beta_2}).$
\end{definition}
Note that
\begin{align*}
    C_{\Ras}(\beta_1\beta_2) &= (C_{\Ras}(\beta_1) \otimes_{Q_{\beta_1}} Q_{\beta_1\beta_2}) \otimes_{Q_{\beta_1\beta_2}} (C_{\Ras}(\beta_2) \otimes_{Q_{\beta_2}} Q_{\beta_1\beta_2})\\
    &\cong C_{\Ras}(\beta_1) \otimes_{Q_{\beta_1}} Q_{\beta_1\beta_2} \otimes_{Q_{\beta_2}} C_{\Ras}(\beta_2)\\
    &\cong C_{\Ras}(\beta_1) \otimes_{Q_{\beta_1}} Q_{\beta_1} \otimes_{\F[\gamma_1, \dots, \gamma_{n-1}]} Q_{\beta_2} \otimes_{Q_{\beta_2}} C_{\Ras}(\beta_2)\\
    &\cong C_{\Ras}(\beta_1) \otimes_{\F[\gamma_1, \dots, \gamma_{n-1}]} C_{\Ras}(\beta_2)
\end{align*}
where the last line closely resembles the tensor product of Rouquier complexes. However, the advantage of having defined $C_{\Ras}(\beta_1\beta_2)$ as we did in Definition \ref{def:tensor product of Rasmussen's complexes} is that it exhibits $C_{\Ras}(\beta_1\beta_2)$ as a chain complex over $Q_{\beta_1\beta_2}$. This allows us to extend the definition of the Rasmussen's double chain complex to an arbitrary braid $\beta$. We also observe that $C_{\Ras}(\beta)$ is a free $Q_\beta$-module.
\begin{lemma}\label{lemma:C free as Q-module}
$C_{\Ras}(\beta)$ is free as a $Q_\beta$-module.
\end{lemma}
\begin{proof}
Indeed, this is clear when $\beta = \sigma_i^{\pm 1}$ is a single crossing from the definition of $C_{\Ras}(\sigma_i^{\pm 1})$. More generally, assume that $C_{\Ras}(\beta_1)$ is a free $Q_{\beta_1}$-module and $C_{\Ras}(\beta_2)$ is a free $Q_{\beta_2}$-module. Then $C_{\Ras}(\beta_1) \otimes_{Q_{\beta_1}} Q_{\beta_1\beta_2}$ and $C_{\Ras}(\beta_2) \otimes_{Q_{\beta_2}} Q_{\beta_1\beta_2}$ are free as $Q_{\beta_1\beta_2}$-modules and hence so is their tensor product $C_{\Ras}(\beta_1\beta_2)$.
\end{proof}
\subsection{Equivalence of constructions}
The chain complex $C_{\Ras}(\beta)$ may be thought of as the counterpart of the Rouquier complex $C_\Rou(\beta)$. A rigorous reinterpretation of this claim leads us to the statement of Theorem \ref{thm:equivalence Ras and Kho}. Several caveats are worth noting.
\begin{enumerate}
    \item Recall that the Rasmussen's chain complex $C_{\Ras}(\beta)$ was defined as a chain complex over $Q_\beta$. However, it can also be interpreted as a chain complex of modules over $\F[\alpha_1, \dots, \alpha_{n-1}, \alpha_1', \dots, \alpha_{n-1}']$, where $\alpha_1, \dots, \alpha_{n-1}$ represent the bottom action and $\alpha_1', \dots, \alpha_{n-1}'$ represent the top action.
    \item Under this interpretation, the chain complexes $C_{\Ras}(\beta)$ and $C_\Rou(\beta)$ are generally \emph{not} chain homotopy equivalent, but are nonetheless isomorphic in the derived category. Taking Hochschild homology of a Rouquier complex involves passing to the derived category through the choice of a free resolution, so the calculation of $\widehat{H}(\overline{\beta})$ with either of $C_\Rou(\beta)$ or $C_{\Ras}(\beta)$ gives the same result.
    \item The complexes are not only derived equivalent. The complex $C_{\Ras}(\beta)$ is free as a chain complex over $\F[\alpha_1, \dots, \alpha_{n-1}, \alpha_1', \dots, \alpha_{n-1}']$ and forms a free resolution of $C_\Rou(\beta)$.
\end{enumerate}
To prove Theorem \ref{thm:equivalence Ras and Kho}, we first remove the assumption that $\beta$ closes to a knot and weaken the conclusion to establish a closely related claim.
\begin{lemma}\label{lemma:isomorphism in derived category}
Let $\beta \in \Br_n$ be a braid. Then $C_{\Ras}(\beta) \cong C_\Rou(\beta)$ in the derived category $D(\F[\alpha_1, \dots, \alpha_{n-1}, \alpha_1', \dots, \alpha_{n-1}'])$.
\end{lemma}
\begin{proof}
We consider the case when $\beta = \sigma_i^{-1}$ is a single negative crossing. Recall that $C_{\Ras}(\sigma_i^{-1})$ is a double chain complex over $Q_{\sigma_i^{-1}} = \F[\alpha_1, \dots, \alpha_{n-1}, \alpha_n]$ where $\alpha_n = \alpha_i'$. As $\F[\alpha_1, \dots, \alpha_{n-1}, \alpha_1', \dots, \alpha_{n-1}']$-modules, we have
$$Q_{\sigma_i^{-1}} \cong \frac{\F[\alpha_1, \dots, \alpha_{n-1}, \alpha_1', \dots, \alpha_{n-1}']}{(\alpha_j' - \alpha_j \text{ for all $j$ with $|i-j| > 1$, } \alpha_{i \pm 1}' - \alpha_{i \pm 1} + \frac{1}{2}(\alpha_i'-\alpha_i))}$$
so $C_{\Ras}(\sigma_i^{-1})$ can be interpreted as a module over $\F[\alpha_1, \dots, \alpha_{n-1}, \alpha_1', \dots, \alpha_{n-1}']$. Recall also that
$$\B_i \cong q^{-1} \frac{\F[\alpha_1, \dots, \alpha_{n-1}, \alpha_1', \dots, \alpha_{n-1}']}{(\alpha_j-\alpha_j' \text{ for $j$ with $|i-j|>1$, } (\alpha_{i\pm 1}+\frac{1}{2}\alpha_i) - (\alpha_{i \pm 1}'+\frac{1}{2}\alpha_i'),  \alpha_i^2-\alpha_i'^2  )}$$
and so the quotient map $Q_{\sigma_i^{-1}} \to \B_i$ gives us a resolution
$$(Q_{\sigma_i^{-1}} \xrightarrow{\alpha_i^2-\alpha_i'^2} Q_{\sigma_i^{-1}}) \to \B_i$$
of $\B_i$ over $\F[\alpha_1, \dots, \alpha_{n-1}, \alpha_1', \dots, \alpha_{n-1}']$. Note that this resolution is the top row of the Rasmussen's complex $C_{\Ras}(\sigma_i^{-1})$. Similar reasoning provides us with a resolution
$$(Q_{\sigma_i^{-1}} \xrightarrow{\alpha_i-\alpha_i'} Q_{\sigma_i^{-1}}) \to R$$
of $R$ over  $\F[\alpha_1, \dots, \alpha_{n-1}, \alpha_1', \dots, \alpha_{n-1}']$, which is the bottom row of the Rasmussen's complex $C_{\Ras}(\sigma_i^{-1})$. Moreover, the Rasmussen's complex $C_{\Ras}(\sigma_i^{-1})$ is a resolution of the Rouquier complex $C_\Rou(\sigma_i^{-1})$ over $\F[\alpha_1, \dots, \alpha_{n-1}, \alpha_1', \dots, \alpha_{n-1}']$ and hence the two are isomorphic in the derived category.

\vspace{1em}

A proof that $C_{\Ras}(\sigma_i) \cong C_\Rou(\sigma_i)$ is similar. Extrapolating from a single crossing to an entire braid requires the following observation. The complex $C_{\Ras}(\sigma_i^{\pm 1})$ is free both as a $\F[\alpha_1, \dots, \alpha_{n-1}]$-module and as a $\F[\alpha_1', \dots, \alpha_{n-1}']$-module. Tensor product and derived tensor product of such modules are isomorphic in the derived category $D(\F[\alpha_1, \dots, \alpha_{n-1}, \alpha_1', \dots, \alpha_{n-1}'])$, establishing the claim.
\end{proof}
\begin{remark}
For $n \geq 3$, $Q_{\sigma_i}$ is not free as a $\F[\alpha_1, \dots, \alpha_{n-1}, \alpha_1', \dots, \alpha_{n-1}']$-module, so $C_{\Ras}(\sigma_i)$ is not a \emph{free} resolution of $C_\Rou(\sigma_i)$. However, this does not contradict the statement of Theorem \ref{thm:equivalence Ras and Kho} since for $n \geq 3$, the closure of $\sigma_i$ is not a knot. Indeed, for $n=2$ when $\overline{\sigma_1}$ is the unknot, we have $Q_{\sigma_1} = \F[\alpha_1, \alpha_1']$ and $C_{\Ras}(\sigma_i)$ is a free resolution of $C_\Rou(\sigma_i)$.
\end{remark}

\begin{lemma}\label{lemma:linear independence}
Let $\beta\in\Br_n$ be a braid with $c$ crossings. Then $\{\alpha_1, \dots, \alpha_{n-1}, \alpha_n^\vee, \dots, \alpha_{n-1+c}^{\vee}\}$ is a basis of the vector space $\F\langle \alpha_1, \dots, \alpha_{n-1+c}\rangle$.
\end{lemma}
\begin{proof}
We proceed by induction. When $c=1$ and $\beta = \sigma_i^{\pm 1}$ is a single crossing, we have $\alpha_n^\vee = \alpha_n$ so $\{\alpha_1, \dots, \alpha_{n-1}, \alpha_n^\vee\} = \{\alpha_1, \dots, \alpha_n\}$ and the claim is trivial.

Assume now the statement is true for a braid $\beta$ with $c$ crossings and consider $\beta' = \beta\sigma_i^{\pm 1}$. Adding a crossing created a new region with $\alpha_{n+c} = \alpha_{n+c}^\vee$ as depicted in Figure \ref{fig:linear independence}.
\begin{figure}[t]
    \begin{tikzpicture}[scale=1]
    \draw[knot] (1,6) to[out=90,in=-90] (2,7);
    \draw[knot] (2,6) to[out=90,in=-90] (1,7);
    \draw[very thick] (0, 6) to (0, 7);
    \draw[very thick] (3, 6) to (3, 7);
    \draw[very thick] (4, 6) to (4, 7);

    \foreach \i in {0, ..., 4}
    {
    \draw[very thick] (\i, 3) to (\i, 3.5);
    \draw[very thick] (\i, 5) to (\i, 6);
    \draw[very thick] (\i, 7) to (\i, 8);
    }

    \draw[] (3.5, 3.45) node[anchor = north]{$\cdots$};
    \draw[very thick] (-0.5, 3.5) rectangle ++(5,1.5);
    \draw[] (2, 4.5) node[anchor = north]{$\beta$};
    \draw[] (0.5, 5.75) node[anchor = north]{$j$};
    \draw[] (1.5, 5.75) node[anchor = north]{$k$};
    \draw[] (2.5, 5.75) node[anchor = north]{$l$};
    \draw[] (3.5, 5.75) node[anchor = north]{$\cdots$};
    \draw[] (1.5, 7.75) node[anchor = north]{$n+c$};
    \end{tikzpicture}  
    \caption{The braid $\beta'$ appearing in the inductive step of the proof of Lemma \ref{lemma:linear independence}. A new crossing $\sigma_i^{-1}$ has been added on top of $\beta$, creating the new region $n+c$ and changing $\alpha_j^\vee$ and $\alpha_l^\vee$. All other original and dual variables remain unchanged.}
    \label{fig:linear independence}
\end{figure}
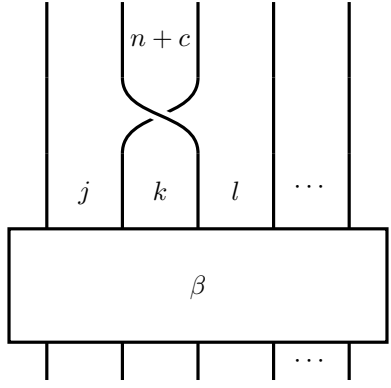

While all original variables are preserved, adding a new crossing changes the meaning of two dual variables. To state these changes explicitly, we have $\alpha_{j}^{\vee, \beta'} = \alpha_{j}^{\vee, \beta} + \frac{1}{2}(\alpha_k^{\vee} - \alpha_{n+c}^\vee)$ and $\alpha_{l}^{\vee, \beta'} = \alpha_{l}^{\vee, \beta} + \frac{1}{2}(\alpha_k^{\vee} - \alpha_{n+c}^\vee)$ where the regions $j$, $k$ and $l$ are as depicted in Figure \ref{fig:linear independence} and the additional superscript indicates whether the dual variable is interpreted in $\beta$ or $\beta'$. Note that if $j, l \leq n-1$, then the dual variables $\alpha_j^{\vee, \beta'}$ and $\alpha_l^{\vee, \beta'}$ do not appear in the set we are trying to argue is a basis, nor do the old dual variables $\alpha_j^{\vee, \beta}$ and $\alpha_l^{\vee, \beta}$ appear in the inductive hypothesis. In this case, we are done immediately since $\alpha_{n+c} = \alpha_{n+c}^\vee$. The rest of our analysis is dedicated to the case in which at least one of the dual variables $\alpha_j^{\vee, \beta}$ and $\alpha_l^{\vee, \beta}$ appears in the inductive hypothesis and hence the corresponding dual variable in $\beta'$ appears in our set of interest. We will show that one can swap $\alpha_j^{\vee, \beta}$ for $\alpha_j^{\vee, \beta'}$ and similarly $\alpha_l^{\vee, \beta}$ for $\alpha_l^{\vee, \beta'}$ without decreasing the dimension of the span of the set of interest. These changes of basis are independent of each other, so we only treat the situation with the region $j$, noting that the situation with region $l$ is completely analogous. We divide our analysis into two cases.
\begin{enumerate}
    \item There is already a crossing between the $i^\text{th}$ and $i+1^\text{st}$ strand in $\beta$. In that case, $\alpha_k^\vee$ is a part of the old basis and hence also of the new set of interest. Since $\alpha_{j}^{\vee, \beta'} = \alpha_{j}^{\vee, \beta} + \frac{1}{2}(\alpha_k^{\vee} - \alpha_{n+c}^\vee)$ and both $\alpha_k^\vee$ and $\alpha_{n+c}^\vee$ are present in the new set of interest, $\alpha_j^{\vee, \beta}$ can be swapped in the basis for $\alpha_j^{\vee, \beta'}$.
    \item There is no crossing between the $i^\text{th}$ and $i+1^\text{st}$ strand in $\beta$. In this case, swapping $\alpha_j^{\vee, \beta}$ for $\alpha_j^{\vee, \beta'}$ is more difficult, because $\alpha_k^\vee$ does not appear in the old basis. Instead, $\alpha_k$ is a part of the old basis and our first step is to express $\alpha_k^\vee$ in terms of the old basis of $\F\langle \alpha_1, \dots, \alpha_{n-1+c}\rangle$ from the inductive hypothesis.
    \begin{claim}
        Let $\{\alpha_1^\circ, \dots, \alpha_{n-1+c}^\circ\}$ be the old basis of  $\F\langle \alpha_1, \dots, \alpha_{n-1+c} \rangle$ where $\alpha_k^\circ = \alpha_k$, $\alpha_j^\circ = \alpha_j^{\vee, \beta}$ and $\alpha_l^\circ \in \{\alpha_l, \alpha_l^\vee\}$ for all other $l$. Express $\alpha_k^\vee = \sum_{j=1}^{n-1+c}c_j\alpha_j^\circ$ for some unique $c_1, \dots, c_{n-1+c} \in \R$. Then $c_k = 1$ and $|c_l| \leq 1$ for all other $l$.
    \end{claim}
    \begin{proof}
        We prove this by flowing from the red dot $\alpha_k^\vee$ to the black dots representing the basis $\{\alpha_1^\circ, \dots, \alpha_{n-1+c}^\circ\}$ as described in Subsection \ref{subsection:Flows}. Since there is no crossing between the $i^\text{th}$ and $i+1^\text{st}$ strand in $\beta$, there is a flowline with flow rate $1$ that flows straight downwards from $\alpha_k^\vee$ to $\alpha_k$. Note also that the only way for a flowline to enter any region is through the crossings near the top and the bottom of the region. However, the region $k$ in $\beta$ is open both on the top and the bottom, so there is no way a flowline can enter it. This means that the straight flowline described above is the only flowline that reaches $\alpha_k$ and hence the coefficient of $\alpha_k$ is $1$ as required. The remaining part of the claim is Lemma \ref{lemma:dual variables as linear combinations}.
    \end{proof}
    We are in a position to show that $\alpha_j^{\vee, \beta}$ can be replaced by $\alpha_j^{\vee, \beta'}$ if all remaining basis elements remain unchanged. Indeed, we can express $\alpha_j^{\vee, \beta'}= \alpha_j^{\vee, \beta} + \frac{1}{2}(\alpha_k^\vee-\alpha_{n+c}^\vee)$ and so $\alpha_j^{\vee, \beta} = \alpha_j^{\vee, \beta'} - \frac{1}{2}(\alpha_k^\vee-\alpha_{n+c}^\vee) = \alpha_j^{\vee, \beta'} -\frac{1}{2}(\sum_{j=1}^{n-1+c} c_j\alpha_j^\circ + c_{n+c}\alpha_{n+c}^\vee)$. Observe that the coefficient of $\alpha_j^\circ = \alpha_j^{\vee, \beta}$ on the right-hand side of the equation is $-\frac{1}{2}c_j \leq \frac{1}{2}$. Most importantly, this coefficient is not equal to $1$, so $\alpha_j^{\vee, \beta}$ can indeed be expressed as a linear combination of the old basis elements and $\alpha_j^{\vee, \beta'}$. This means that $\alpha_j^{\vee, \beta}$ can be swapped in the basis for $\alpha_j^{\vee, \beta'}$.
\end{enumerate}
The inductive step is complete in either case, thus establishing the lemma. 
\end{proof}

We now strengthen the result slightly.
\begin{lemma}\label{lemma:algebraic independence}
Let $\beta \in \Br_n$ be a braid with $c$ crossings. Then $\{\alpha_1, \dots, \alpha_{n-1}, \alpha_n^\vee, \dots, \alpha_{n-1+c}^{\vee}\}$ $\subset Q_\beta$ is algebraically independent over $\F$.
\end{lemma}
\begin{proof}
By Lemma \ref{lemma:linear independence}, we have that $\{\alpha_1, \dots, \alpha_{n-1}, \alpha_n^\vee, \dots, \alpha_{n-1+c}^\vee\}$ is a basis of $\F\langle \alpha_1, \dots, \alpha_{n-1+c}\rangle$ and hence $\F[\alpha_1, \dots, \alpha_{n-1}, \alpha_n^\vee, \dots, \alpha_{n-1+c}^\vee] = \F[\alpha_1, \dots, \alpha_{n-1+c}]$. Passing to the field of fractions $\F(\alpha_1, \dots, \alpha_{n-1+c})$, we observe that the transcendence degree of the field extension $\F(\alpha_1, \dots, \alpha_{n-1+c})/\F$ is $n-1+c$, which is the same as the number of generators of $\F[\alpha_1, \dots, \alpha_{n-1}, \alpha_n^\vee, \dots, \alpha_{n-1+c}^\vee]$ as an $\F$-algebra. Therefore $\alpha_1^\circ, \dots, \alpha_{n-1+c}^\circ$ are algebraically independent.
\end{proof}

\begin{proof}[Proof of Theorem \ref{thm:equivalence Ras and Kho}]
By Lemma \ref{lemma:isomorphism in derived category}, the only outstanding step is to show that when $\overline{\beta}$ is a knot, $C_{\Ras}(\beta)$ is free over $\F[\alpha_1, \dots, \alpha_{n-1}, \alpha_1', \dots, \alpha_{n-1}']$. In the presence of this assumption, the regions near the bottom and the top of $\beta$ are different. By Lemma \ref{lemma:algebraic independence}, $\{\alpha_1, \dots, \alpha_{n-1}, \alpha_n^\vee, \dots, \alpha_{n-1+c}^{\vee}\}$ are algebraically independent over $\F$. This is a stronger result than is needed, in the sense that discarding some of the variables gives the algebraic independence of $\{\alpha_1, \dots, \alpha_{n-1}, \alpha_1', \dots, \alpha_{n-1}'\}$. Therefore, there exists an embedding $\F[\alpha_1, \dots, \alpha_{n-1}, \alpha_1', \dots, \alpha_{n-1}'] \hookrightarrow Q_\beta$. Since $C_{\Ras}(\beta)$ is free over $Q_\beta$ by Lemma \ref{lemma:C free as Q-module}, it is also free over $\F[\alpha_1, \dots, \alpha_{n-1}, \alpha_1', \dots, \alpha_{n-1}']$ as required.
\end{proof}

Theorem \ref{thm:equivalence Ras and Kho} explains the usefulness of Rasmussen's construction. In the process of calculating $\widehat{\mathit{H}}(\overline{\beta})$, we no longer need to pass to an arbitrary free resolution of $C_\Rou(\beta)$ since $C_{\Ras}(\beta)$ provides one for us canonically.

\section{Higher differentials}\label{sec:higher differentials}
Rasmussen's version possesses an additional advantage over Khovanov's version: it can be used to construct spectral sequences from Khovanov-Rozansky homology to $\mathfrak{sl}_k$ homologies for all $k \in \N$. To define them, we will enhance Rasmussen's double chain complexes associated to crossings to triple complexes
$$
C_{\Ras}(\sigma_i)=
\begin{tikzcd}
    Q_{\sigma_i} \arrow[rr, shift left, "\alpha_i-\alpha_i'"] && Q_{\sigma_i} \arrow[ll, blue, shift left, "d_k"]\\
    &&\\
    Q_{\sigma_i} \arrow[uu, "\alpha_i+\alpha_i'"] \arrow[rr, shift left, "\alpha_i^2-\alpha_i'^2"] && Q_{\sigma_i} \arrow[ll, blue, shift left, "d_k"] \arrow[uu, "1", swap]
\end{tikzcd}
\text{ and }
C_{\Ras}(\sigma_i^{-1}) = 
\begin{tikzcd}
    Q_{\sigma_i^{-1}} \arrow[rr, shift left, "\alpha_i^2-\alpha_i'^2"] && Q_{\sigma_i^{-1}} \arrow[ll, blue, shift left, "d_k"]\\
    &&\\
    Q_{\sigma_i^{-1}} \arrow[uu, "1"] \arrow[rr, shift left, "\alpha_i-\alpha_i'"] && Q_{\sigma_i^{-1}} \arrow[ll, blue, shift left, "d_k"] \arrow[uu, "\alpha_i+\alpha_i'", swap]
\end{tikzcd}
$$
by adding the maps $d_k$ drawn in blue. The chain map on $\widehat{\mathit{H}}(K)$ induced by $d_k$ will represent the differential on the first page of the spectral sequence to $\mathfrak{sl}_k$.

The map $d_k$ has a nontrivial polynomial expression in terms of the variables $\alpha_1, \dots, \alpha_{n-1}, \alpha_1', \dots, \alpha_{n-1}' \in Q_{\sigma_i^{\pm 1}}$, defined as follows. Consider the expression
$$W_k := \frac{1}{n^{k+1}} \sum_{l=1}^n \left(\left(\sum_{j=1}^{n-1} j\alpha_j - \sum_{j=l}^{n-1}n\alpha_j\right)^{k+1} - \left(\sum_{j=1}^{n-1} j\alpha_j' - \sum_{j=l}^{n-1}n\alpha_j'\right)^{k+1}\right) \in Q_{\sigma_i^{\pm 1}},$$
where we treat $Q_{\sigma_i^{\pm 1}}$ as an $\F[\alpha_1, \dots, \alpha_{n-1}, \alpha_1', \dots, \alpha_{n-1}']$-module as in the previous section. The map $d_k$ in Rasmussen's complexes $C_{\Ras}(\sigma_i^{\pm 1})$ is the unique chain map satisfying $(d_+ + d_k)^2 = W_k$. In other words, we now show that $W_k$ is divisible by both $\alpha_i-\alpha_i'$ and $\alpha_i+\alpha_i'$ and then set the leftward maps in the two diagrams equal to $\frac{W_k}{\alpha_i-\alpha_i'}$ or $\frac{W_k}{\alpha_i^2-\alpha_i'^2}$.

\begin{lemma}
For any $n, k \in \N$ and $i \in \{1, \dots, n-1\}$, the quantity $W_k$ is divisible by $\alpha_i^2-\alpha_i'^2$.
\end{lemma}
\begin{proof}

Let us first assume $\alpha_i' = \alpha_i$. Then also $\alpha_{i \pm 1}' = \alpha_{i \pm 1} + \frac{1}{2}\alpha_i - \frac{1}{2}\alpha_i' = \alpha_{i \pm 1}$ in $Q_{\sigma_i^{\pm 1}}$. Since $\alpha_j' = \alpha_j$ in $Q_{\sigma_i^{\pm 1}}$ for all $j$ with $|i-j| > 1$, it follows that the expressions inside the inner pairs of parentheses in $W_k$ coincide and thus $W_k$ evaluates to $0$. Therefore $\alpha_i-\alpha_i'$ divides $W_k$ as required.

Similarly, if $\alpha_i' = -\alpha_i$, then $\alpha_{i \pm 1}' = \alpha_{i \pm 1} + \alpha_i$ and we still have $\alpha_j' = \alpha_j$ in $Q_{\sigma_i^{\pm 1}}$ for all $j$ with $|i-j| > 1$. By carefully inspecting the outer sum, we will show that most of its terms vanish. We have two cases, each of which splits further into two subcases.
\begin{enumerate}
    \item $i \in \{1, \dots, n-2\}$. It follows that $\sum_{j=1}^{n-1}j\alpha_j = \sum_{j=1}^{n-1}j\alpha_j'$ by a simple calculation.
\begin{enumerate}
    \item If $l \notin \{i, i+1\}$, then either $l \leq i-1, i, i+1 \leq n-1$ or $i-1, i, i+1 < l$ and it can be verified that in both cases $\sum_{j=l}^{n-1} n\alpha_j = \sum_{j=l}^{n-1} n\alpha_j'$. Therefore $\left(\sum_{j=1}^{n-1} j\alpha_j - \sum_{j=l}^{n-1}n\alpha_j\right)^{k+1} = \left(\sum_{j=1}^{n-1} j\alpha_j' - \sum_{j=l}^{n-1}n\alpha_j'\right)^{k+1}$ and so the summand corresponding to $l$ vanishes.
    \item If $l \in \{i, i+1\}$, then the summand corresponding to $l$ does not vanish by itself. However, we will show that the summands corresponding to $i$ and $i+1$ cancel each other. To see this, we can verify through an easy calculation that $\sum_{j=i}^{n-1}n\alpha_j' = \sum_{j=i+1}^{n-1}n\alpha_j$ and similarly $\sum_{j=i+1}^{n-1}n\alpha_j' = \sum_{j=i}^{n-1}n\alpha_j$. It is now apparent that the terms corresponding to $i, i+1$ are negatives of each other.
\end{enumerate}
    \item $i = n-1$.
    \begin{enumerate}
        \item If $l \notin \{n-1, n\}$, then $l \leq n-2$ and we can calculate that $\sum_{j=1}^{n-1}j\alpha_j' = \sum_{j=1}^{n-3} j\alpha_j + (n-2)(\alpha_{n-2}+\alpha_{n-1}) + (n-1)(-\alpha_{n-1}) = \sum_{j=1}^{n-1}j\alpha_j - n\alpha_{n-1}$ and similarly $\sum_{j=1}^{n-1}n\alpha_j' = \sum_{j=1}^{n-1}n\alpha_j+n(\alpha_{n-2}+\alpha_{n-1})+n(-\alpha_{n-1}) = \sum_{j=1}^{n-1}n\alpha_j - n\alpha_{n-1}$. Looking at the difference of these two sums, we see that $\sum_{j=1}^{n-1} j\alpha_j' - \sum_{j=l}^{n-1}n\alpha_j' =\sum_{j=1}^{n-1} j\alpha_j - \sum_{j=l}^{n-1}n\alpha_j$. By raising the equation to the $(k+1)^{\text{st}}$ power, we see that the summand corresponding to $l$ vanishes.
        \item As above, if $l \in \{n-1, n\}$, then the summand corresponding to $l$ does not vanish by itself -- instead, those two summands cancel each other. This can be verified with a straightforward manual calculation. The summand corresponding to $l = n-1$ is $(\sum_{j=1}^{n-1}j\alpha_j - n\alpha_{n-1})^{k+1} - (\sum_{j=1}^{n-1}j\alpha_j)^{k+1}$ and the summand corresponding to $l = n$ can be seen to equal $(\sum_{j=1}^{n-1}j\alpha_j)^{k+1} - (\sum_{j=1}^{n-1}j\alpha_j - n\alpha_{n-1})^{k+1}$.
    \end{enumerate}
\end{enumerate}
In either case, $W_k$ evaluates to $0$ and so $\alpha_i+\alpha_i'$ divides $W_k$ as required.
\end{proof}
This shows that $d_k$ is a well-defined map on a Rasmussen's chain complex $C_{\Ras}(\sigma_i^{\pm 1})$ associated to a single crossing. The definition extends to arbitrary braids inductively with the use of the Leibniz formula, \emph{i.e.} $d_k(x \otimes y) = d_k(x) \otimes y + x \otimes d_k(y)$.

\begin{example}
Let $n = 2$ and $k \in \N$. Then $d_k$ can be written down explicitly. We have
\begin{align*}
W_k &= \frac{1}{2^{k+1}}\left(\left(\left(-\alpha_1\right)^{k+1}+\alpha_1^{k+1}\right)- \left(\left(-\alpha_1'\right)^{k+1}+\alpha_1'^{k+1}\right)\right)\\
&=\frac{1}{2^{k+1}}\left((-1)^{k+1}+1\right)\left(\alpha_1^{k+1}-\alpha_1'^{k+1}\right)\\
&= \begin{cases}
    \frac{1}{2^k}(\alpha_1^{k+1}-\alpha_1'^{k+1}) & | \ k \text{ odd}\\
    0 & | \ k \text{ even.}
\end{cases}
\end{align*}
Therefore $d_k = 0$ for all even $k$. One interesting corollary of the case $k = 2$ is that the reduced Khovanov-Rozansky and $\mathfrak{sl}_2$/Khovanov homologies of $\overline{\sigma_1^{2m+1}} = T_{2, 2m+1}$ are isomorphic for all $m \in \N$, \emph{i.e.} $\widehat{\mathit{H}}(T_{2, 2m+1}) \cong \mathit{Kh}(T_{2, 2m+1})$. This further implies that $d_k = 0$ for all $k \geq 2$ on $\widehat{\mathit{H}}(T_{2, 2m+1})$, despite the fact that $d_k \neq 0$ for odd $k$ on the level of $C_{\Ras}(\sigma_1^{2m+1})$.
\end{example}
\begin{example}
Let $n = 3$ and $k = 1$. We calculate
\begin{align*}
W_1 &= \frac{1}{3^2}\left( (-2\alpha_1-\alpha_2)^2 + (\alpha_1-\alpha_2)^2 + (\alpha_1+2\alpha_2)^2 \right) \\
&\quad - \frac{1}{3^2} \left( (-2\alpha_1'-\alpha_2')^2 + (\alpha_1'-\alpha_2')^2 + (\alpha_1'+2\alpha_2')^2 \right).
\end{align*}
If $i = 1$, then we have $\alpha_2' = \alpha_2 + \frac{1}{2}(\alpha_1-\alpha_1')$ in $Q_{\sigma_1}$ and $W_1$ evaluates to $\frac{1}{2} (\alpha_1^2-\alpha_1'^2)$. Note that this implies $d_1$ has the same form on two and three-stranded braids.
\end{example}
\begin{example}
Let $n = 3$ and $k = 2$. The computational details quickly become messy as $k$ increases, so we omit some steps in this example. An explicit calculation shows that $\sum_{l=1}^3(\sum_{j=1}^{2}j\alpha_j-\sum_{j=l}^{2}3\alpha_j)^3$ simplifies to $-6\alpha_1^3+6\alpha_2^3-9\alpha_1^2\alpha_2+9\alpha_1\alpha_2^2$. If $i = 1$, then $\alpha_2'=\alpha_2+\frac{1}{2}(\alpha_1-\alpha_1')$ in $Q_{\sigma_1}$ as in the previous example and some calculation shows that $W_2 = -\frac{1}{4}(\alpha_1^2-\alpha_1'^2)(\alpha_1+2\alpha_2)$. This indeed confirms that $\alpha_1^2-\alpha_1'^2$ divides $W_2$ and gives us an explicit form of $d_2$, the differential on the first page of the spectral sequence from Khovanov-Rozansky to $\mathfrak{sl}_2$/Khovanov homology.
\end{example}

\section{Dot-sliding homotopies}\label{sec:dot-sliding homotopies}
This section introduces the maps $\zeta_{\alpha_j}$, the analogues of the dot-sliding homotopies $\xi_{j}$ considered by Gorsky and Hogancamp \cite{gorsky2022hilbert}.

\subsection{Map \texorpdfstring{$f_\beta$}{fb}}
Let $\beta = \prod_{k=1}^c \sigma_{i_k}^{a_k} \in \Br_n$ be a braid. In this subsection, we construct a map $f_\beta: \F[\alpha_1, \dots, \alpha_{n-1}] \to \F[\alpha_1', \dots, \alpha_{n-1}']$ that sends the variables associated to the regions on the bottom of $\beta$ to their primed counterparts. As has been a recurrent theme throughout the paper, the construction of $f_\beta$ is inductive. If $c=0$, \emph{i.e.} $\beta = 1 \in \Br_n$, then $f_1(\alpha_j) = \alpha_j'$ for all $j \in \{1, \dots, n-1\}$. If $c=1$, \emph{i.e.} $\beta = \sigma_i^{\pm 1}$ is a single crossing, then $f_{\sigma_i^{\pm 1}}: \F[\alpha_1, \dots, \alpha_{n-1}] \to \F[\alpha_1', \dots, \alpha_{n-1}']$ is given by the unique extension of
\begin{align*}
    &\alpha_i \mapsto -\alpha_i',\\
    &\alpha_{i\pm 1} \mapsto \alpha_{i\pm 1}' + \alpha_i', \text{ and}\\
    &\alpha_j \mapsto \alpha_j' \text{ for $j$ with $|i-j|>1$.}
\end{align*}
to a ring homomorphism. Let now $\beta = \beta_1\beta_2$ be a composition of two braids. Then define $f_\beta = f_{\beta_2}f_{\beta_1}$ as the composition of $f_{\beta_1}$ and $f_{\beta_2}$ as well. Note that this is well-defined since the primed variables of $\beta_1$ are identified with the bottom variables of $\beta_2$ in the tensor product $Q_{\beta_1} \otimes_{\F[\gamma_1, \dots, \gamma_{n-1}]} Q_{\beta_2}$.
\begin{lemma}
The map $f_\beta$ is well-defined, \emph{i.e.} it does not depend on the expression of $\beta$ as a product of braids.
\end{lemma}
\begin{proof}
We need to show that $f_\beta$ is invariant under the two braid Reidemeister moves. The verification of the first braid Reidemeister move amounts to showing that $f_{\sigma_i\sigma_i^{-1}} = f_{\sigma_i^{-1}\sigma_i} = f_1$. This check can be performed manually and is essentially the same for $f_{\sigma_i\sigma_i^{-1}}$ and $f_{\sigma_i^{-1}\sigma_i}$, so we only include the details for $f_{\sigma_i\sigma_i^{-1}}$. Let $\gamma_1, \dots, \gamma_{n-1}$ be the variables representing the top action on $Q_{\sigma_i}$ or equivalently the bottom action on $Q_{\sigma_i^{-1}}$. Then the map $f_{\sigma_i\sigma_i^{-1}}$ is given by $\alpha_i \mapsto -\gamma_i  \mapsto \alpha_i'$, $\alpha_{i\pm 1} \mapsto \gamma_{i\pm 1}+\gamma_i \mapsto (\alpha_{i \pm 1}' + \alpha_i') - \alpha_i' = \alpha_{i \pm 1}'$, and $\alpha_j \mapsto \gamma_j \mapsto \alpha_j'$ for all $j$ with $|i-j|>1$. In these expressions, the first $\mapsto$ denotes an application of $f_{\sigma_i}$ and the second $\mapsto$ denotes an application of $f_{\sigma_i^{-1}}$. Therefore, $f_{\sigma_i\sigma_i^{-1}}(\alpha_j) = \alpha_j'$ for all $j \in \{1, \dots, n-1\}$ and so $f_{\sigma_i\sigma_i^{-1}}=f_1$ as required.

We perform a similar calculation to establish that $f_{\sigma_i\sigma_{i+1}\sigma_i} = f_{\sigma_{i+1}\sigma_i\sigma_{i+1}}$. Let us first calculate $f_{\sigma_i\sigma_{i+1}\sigma_i}$. Let $\gamma_1, \dots, \gamma_{n-1}$ be the variables representing the top action on $\sigma_i$ and the bottom action on $\sigma_{i+1}$ and let $\delta_1, \dots, \delta_{n-1}$ be the variables representing the top action on $\sigma_{i+1}$ and the bottom action on $\sigma_i$. Then
\begin{align*}
&\alpha_{i-1} \mapsto \gamma_{i-1}+\gamma_i \mapsto \delta_{i-1}+\delta_i+\delta_{i+1} \mapsto \alpha_{i-1}' + \alpha_{i}'+\alpha_{i+1}'\\
&\alpha_i \mapsto -\gamma_i \mapsto -(\delta_i+\delta_{i+1}) \mapsto -(-\alpha_i'+\alpha_{i+1}'+\alpha_i')=-\alpha_{i+1}'\\
&\alpha_{i+1} \mapsto \gamma_{i+1}+\gamma_i \mapsto -\delta_{i+1}+\delta_i+\delta_{i+1} = \delta_{i} \mapsto -\alpha_{i}'\\
&\alpha_{i+2} \mapsto \gamma_{i+2} \mapsto \delta_{i+2}+\delta_{i+1} \mapsto \alpha_{i+2}' + \alpha_{i+1}' + \alpha_{i}'
\end{align*}
and $\alpha_j \mapsto \gamma_j \mapsto \delta_j \mapsto \alpha_j'$ for all $j$ with $|i-j|>1$. It is not difficult to see that a similar calculation for $f_{\sigma_{i+1}\sigma_i\sigma_{i+1}}$ gives the same result.
\end{proof}
In fact, since $f_{\sigma_i} = f_{\sigma_i^{-1}}$, we have that $f_\beta$ only depends on the underlying permutation of $\beta$ and not on the braid itself.

\begin{remark}
    The map $f_{\beta}$ has a geometric interpretation. Every bottom region between two strands is sent to the signed top region between the same two strands. For example, if $\beta = \sigma_i$ is a single crossing, then $\alpha_{i+i} \mapsto \alpha_i'+\alpha_{i+1}'$ since the region $\alpha_{i+1}$ between the $i+1^{\text{st}}$ and $i+2^{\text{nd}}$ strands ``expands'' to the region between the $i^{\text{th}}$ and $i+2^{\text{nd}}$ regions on top. Crossing of two strands introduces a minus sign.
\end{remark}

\subsection{Maps \texorpdfstring{$\zeta_\alpha$}{zetaa}}
Let $\beta \in \Br_n$. We now construct a family of maps $\zeta_{\alpha_j}^\beta: C(\beta) \to C(\beta)$ parametrized by $j \in \{1, \dots, n-1\}$. The maps $\zeta_{\alpha_j}^\beta$ should be thought of as the analogues of the maps $\xi_j$ studied in \cite{gorsky2022hilbert}. When there is no ambiguity regarding the underlying braid $\beta$, we typically omit the superscript in $\zeta_{\alpha_j}^\beta$ to aid clarity.

\vspace{1em}

We begin by defining the maps $\zeta_{\alpha_j}$ on $\beta = \sigma_i^{\pm 1}$ and eventually extend the construction to arbitrary braids. Consider first the case $\beta = \sigma_i$. Then $\zeta_{\alpha_j} = 0$ for all $j$ with $|i-j| > 1$. The remaining three maps $\zeta_{\alpha_{i-1}}$, $\zeta_{\alpha_i}$, and $\zeta_{\alpha_{i+1}}$ are nonzero and defined in terms of Rouquier complexes as the green downwards maps in
$$
\begin{tikzcd}
    \B_i \arrow[d, dark green, shift left, "-\frac{1}{2}"]\\
    R \arrow[u, shift left, "\alpha_i+\alpha_i'"]
\end{tikzcd}
,
\quad
\begin{tikzcd}
    \B_i \arrow[d, dark green, shift left, "1"]\\
    R \arrow[u, shift left, "\alpha_i+\alpha_i'"]
\end{tikzcd}
\quad
\text{ and }
\quad
\begin{tikzcd}
    \B_i \arrow[d, dark green, shift left, "-\frac{1}{2}"]\\
    R \arrow[u, shift left, "\alpha_i+\alpha_i'"]
\end{tikzcd}
$$
respectively. Note that $\zeta_{\alpha_{i-1}} = \zeta_{\alpha_{i+1}}$. The same idea can be expressed in terms of Rasmussen's complexes as the maps $\zeta_{\alpha_{i-1}}$, $\zeta_{\alpha_i}$, and $\zeta_{\alpha_{i+1}}: C_{\Ras}(\sigma_i) \to C_{\Ras}(\sigma_i)$ given by the green downwards maps in the commutative squares
$$
\begin{tikzcd}
    Q_{\sigma_i} \arrow[rr, "\alpha_i^2-\alpha_i'^2"] \arrow[dd, dark green, shift right, swap, "\alpha_i+\alpha_i'"] && Q_{\sigma_i} \arrow[dd, dark green, shift left, "1"]\\
    &&\\
    Q_{\sigma_i} \arrow[uu, shift right, swap, "1"] \arrow[rr, "\alpha_i-\alpha_i'"] && Q_{\sigma_i} \arrow[uu, shift left, "\alpha_i+\alpha_i'"]
\end{tikzcd}
\text{ for $\zeta_{\alpha_i}$ and }
\begin{tikzcd}
    Q_{\sigma_i} \arrow[rr, "\alpha_i^2-\alpha_i'^2"] \arrow[dd, dark green, shift right, swap, "-\frac{\alpha_i+\alpha_i'}{2}"] && Q_{\sigma_i} \arrow[dd, dark green, shift left, "-\frac{1}{2}"]\\
    &&\\
    Q_{\sigma_i} \arrow[uu, shift right, swap, "1"] \arrow[rr, "\alpha_i-\alpha_i'"] && Q_{\sigma_i} \arrow[uu, shift left, "\alpha_i+\alpha_i'"]
\end{tikzcd}
\text{ for $\zeta_{\alpha_{i-1}} = \zeta_{\alpha_{i+1}}.$}
$$
For a linear combination $\alpha = \sum_{j=1}^{n-1} \lambda_j\alpha_j$ with $\lambda_1, \dots, \lambda_{n-1} \in \F$, define $\zeta_{\alpha}$ to be the corresponding linear combination of the maps $\zeta_{\alpha_j}$, \emph{i.e.}  $\zeta_{\alpha} = \sum_{j=1}^{n-1} \lambda_j \zeta_{\alpha_j}$. With that convention, we have defined the maps $\zeta_\alpha: C(\sigma_i) \to C(\sigma_i)$ on chain complexes associated to positive crossings.

Let now $\beta = \sigma_i^{-1}$ be a negative crossing for some $i \in \{1, \dots, n-1\}$. As before, we have that $\zeta_{\alpha_j} = 0$ for all $j \in \{1, \dots, n-1\}$. The maps $\zeta_{\alpha_{i-1}}$, $\zeta_{\alpha_{i}}$, and $\zeta_{\alpha_{i+1}}$ are nonzero and defined in terms of Rouquier complexes as the downwards maps in
$$
\begin{tikzcd}
    R \arrow[d, dark green, shift left, "-\frac{\alpha_i+\alpha_i'}{2}"]\\
    \B_i \arrow[u, shift left, "1"]
\end{tikzcd}
,
\quad
\begin{tikzcd}
    R \arrow[d, dark green, shift left, "\alpha_i+\alpha_i'"]\\
    \B_i \arrow[u, shift left, "1"]
\end{tikzcd}
\quad
\text{ and }
\quad
\begin{tikzcd}
    R \arrow[d, dark green, shift left, "-\frac{\alpha_i+\alpha_i'}{2}"]\\
    \B_i \arrow[u, shift left, "1"]
\end{tikzcd}
$$
respectively. Alternatively, the maps $\zeta_\alpha$ can also be expressed in the language of Rasmussen's complexes as the downwards maps in 
$$
\begin{tikzcd}
    Q_{\sigma_i^{-1}} \arrow[rr, "\alpha_i-\alpha_i'"] \arrow[dd, dark green, shift right, swap, "1"] && Q_{\sigma_i^{-1}} \arrow[dd, dark green, shift right, swap, "\alpha_i+\alpha_i'"]\\
    &&\\
    Q_{\sigma_i^{-1}} \arrow[uu, shift right, swap, "\alpha_i+\alpha_i'"] \arrow[rr, "\alpha_i^2-\alpha_i'^2"]  && Q_{\sigma_i^{-1}} \arrow[uu, shift right, swap, "1"]
\end{tikzcd}
\text{ for $\zeta_{\alpha_i}$ and }
\begin{tikzcd}
    Q_{\sigma_i^{-1}} \arrow[rr, "\alpha_i-\alpha_i'"] \arrow[dd, dark green, shift right, swap, "-\frac{1}{2}"] && Q_{\sigma_i^{-1}}\arrow[dd, dark green, shift right, swap, "-\frac{\alpha_i+\alpha_i'}{2}"]\\
    &&\\
    Q_{\sigma_i^{-1}} \arrow[uu, shift right, swap, "\alpha_i+\alpha_i'"] \arrow[rr, "\alpha_i^2-\alpha_i'^2"]  && Q_{\sigma_i^{-1}} \arrow[uu, shift right, swap, "1"]
\end{tikzcd}
\text{ for $\zeta_{\alpha_{i-1}} = \zeta_{\alpha_{i+1}}.$}
$$
We also extend the maps $\zeta_{\alpha_j}$ to general linear combinations $\zeta_\alpha:C(\sigma_i^{-1}) \to C(\sigma_i^{-1})$.

Having defined the maps $\zeta_\alpha: C(\sigma_i^{\pm 1}) \to C(\sigma_i^{\pm 1})$ on single crossings, we may extend them to arbitrary braids $\beta$ inductively via the Leibniz rule. If $\beta = \beta_1\beta_2$ is a composition of braids, then recall that $C(\beta_1\beta_2) =  C(\beta_1) \otimes_{\F[\gamma_1, \dots, \gamma_{n-1}]} C(\beta_2)$ and define $\zeta_\alpha^\beta: C(\beta_1\beta_2) \to  C(\beta_1\beta_2)$ via $\zeta_\alpha^\beta = \zeta_\alpha^{\beta_1} \otimes \id + \id \otimes \zeta_{f_{\beta_1}(\alpha)}^{\beta_2}$. More explicitly, if $x \in C(\beta_1)$ and $y \in C(\beta_2)$, then $\zeta_{\alpha}^\beta(x \otimes y) = \zeta_{\alpha}^{\beta_1}(x) \otimes y + x \otimes \zeta_{f_{\beta_1}(\alpha)}^{\beta_2}(y)$.
\begin{lemma}\label{lemma:[d,zeta]}
Let $\beta \in \Br_n$ be a braid and $\alpha$ a linear combination of $\alpha_1, \dots, \alpha_{n-1}$. Then $[d, \zeta_\alpha] = \alpha - f_\beta(\alpha)$. 
\end{lemma}
\begin{proof}
Since both sides of the equation are linear in $\alpha$, it suffices to establish the claim for $\alpha \in \{\alpha_1, \dots, \alpha_{n-1}\}$. Begin with $\beta = \sigma_i^{\pm 1}$. If $\alpha = \alpha_j$ for $j$ with $|i-j|>1$, then $\zeta_\alpha = 0$ so $[d, \zeta_\alpha] = 0$ and also $f_{\sigma_i^{\pm 1}}(\alpha) = \alpha'$. Hence $\alpha - f_{\sigma_i^{\pm 1}}(\alpha) = \alpha - \alpha' = 0$. If $\alpha = \alpha_i$, then $[d, \zeta_\alpha] = \alpha_i + \alpha_i' = \alpha_i - f_{\sigma_i^{\pm 1}}(\alpha_i)$ as required. Finally, if $\alpha = \alpha_{i \pm 1}$, then $\alpha-f_{\sigma_i^{\pm 1}}(\alpha) = \alpha_{i\pm 1} - (\alpha_{i\pm 1}'+\alpha_i') = \alpha_{i\pm 1} - (\alpha_{i\pm 1} + \frac{1}{2}(\alpha_i-\alpha_i') + \alpha_i') = -\frac{\alpha_i+\alpha_i'}{2} = [d, \zeta_\alpha]$ as required. This established the claim whenever $\beta$ is a single crossing. More generally, let $\beta = \beta_1\beta_2$. Then
\begin{align*}
[d, \zeta_\alpha] &= [d\otimes\id+\id\otimes d, \zeta_\alpha \otimes \id + \id \otimes \zeta_{f_{\beta_1}(\alpha)}]\\
&= [d\otimes \id, \zeta_\alpha \otimes \id] + [d\otimes \id, \id \otimes \zeta_{f_{\beta_1}(\alpha)}] + [ \id\otimes d, \zeta_\alpha \otimes \id] + [\id\otimes d, \id \otimes \zeta_{f_{\beta_1}(\alpha)}]\\
&= [d, \zeta_\alpha] \otimes \id + 0 + 0 + \id \otimes [d, \zeta_{f_{\beta_1}(\alpha)}]\\
&= \alpha - f_{\beta_1}(\alpha) + f_{\beta_1}(\alpha) - f_{\beta_2}(f_{\beta_1}(\alpha))\\
&= \alpha - f_\beta(\alpha)
\end{align*}
by induction.
\end{proof}

\section{Connected sum}
This section describes how Khovanov-Rozansky homology behaves under connected sum of knots. Our proof follows \cite[Lemma 7.8]{rasmussen2016some}.
\begin{theorem}\label{thm:connected sum}
$\widehat{\mathit{H}}(K_1 \# K_2) \cong \widehat{\mathit{H}}(K_1) \otimes_\F \widehat{\mathit{H}}(K_2)$.
\end{theorem}
\begin{proof}
Let $\beta_1 \in \Br_{n_1}$ and $\beta_2 \in \Br_{n_2}$ be braids with $c_1$ and $c_2$ crossings respectively such that $K_1 = \overline{\beta_1}$ and $K_2 = \overline{\beta_2}$. Consider the knot diagram $D$ of $K_1 \# K_2$ obtained as a closure of an $n_1 + n_2 -1$-strand braid $\beta$ as depicted in Figure \ref{fig:connected sum}.
\begin{figure}[t]
    \begin{tikzpicture}[scale=0.9]
    \draw[very thick] (0, 0) to (0, 3.5);
    \draw[very thick] (1, 0) to (1, 3.5);
    \draw[very thick] (2, 0) to (2, 3.5);
    \draw[very thick] (3, 0) to (3, 1);
    \draw[very thick] (4, 0) to (4, 1);
    \draw[very thick] (5, 0) to (5, 1);
    \draw[very thick] (2.5, 1) rectangle ++(3,1.5);
    \draw[very thick] (-0.5, 3.5) rectangle ++(4,1.5);

    \draw[very thick] (0, 5) to (0, 6);
    \draw[very thick] (1, 5) to (1, 6);
    \draw[very thick] (2, 5) to (2, 6);
    \draw[very thick] (3, 5) to (3, 6);
    \draw[very thick] (3, 2.5) to (3, 3.5);
    \draw[very thick] (4, 2.5) to (4, 6);
    \draw[very thick] (5, 2.5) to (5, 6);

    \draw[] (4,2) node[anchor = north]{$\beta_1$};
    \draw[] (1.5, 4.5) node[anchor = north]{$\beta_2$};

    \draw[] (3.5, 0.75) node[anchor = north]{$\alpha_{1, 1}$};
    \draw[] (4.5, 0.75) node[anchor = north]{$\alpha_{1, 2}$};
    \draw[] (0.5, 3.25) node[anchor = north]{$\alpha_{2, 1}$};
    \draw[] (1.5, 3.25) node[anchor = north]{$\alpha_{2, 2}$};
    \draw[] (2.5, 3.25) node[anchor = north]{$\alpha_{2, 3}$};
    \end{tikzpicture}  
    \caption{A $n_1+n_2-1$-strand braid $\beta$ appearing in the proof of Theorem \ref{thm:connected sum}. Its closure is a knot diagram of $\overline{\beta_1} \# \overline{\beta_2}$.}
    \label{fig:connected sum}
\end{figure}
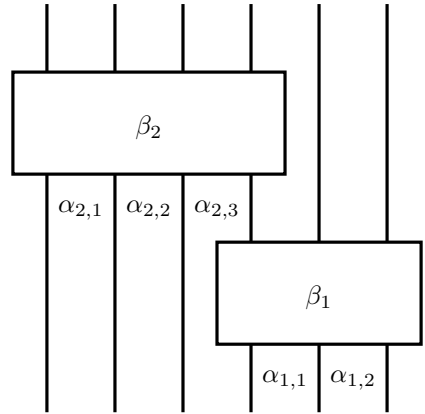

Label all regions of $D$ such that $\alpha_{1,1}, \dots, \alpha_{1,n_1-1+c_1}$ are the regions on the ``$\beta_1$ side'' of the diagram and $\alpha_{2,1}, \dots, \alpha_{2, n_2-1+c_2}$ are the regions on the ``$\beta_2$ side'' of the diagram. Therefore $Q_\beta = \F[\alpha_{1, 1}, \dots, \alpha_{1, n_1-1+c_1}, \alpha_{2, 1}, \dots, \alpha_{2, n_2-1+c_2}]$ by Definition \ref{def:Rasmussen's ring}. We can write $\beta = (\id_{n_2-1} \sqcup \beta_1) (\beta_2 \sqcup \id_{n_1-1})$, where $\id_n$ denotes the identity braid on $n$ strands. Consider now the Rasmussen's rings and chain complexes of $\id_{n_2-1} \sqcup \beta_1$ and $\beta_2 \sqcup \id_{n_1-1}$. We have $Q_{\id_{n_2-1} \sqcup \beta_1} = \F[\alpha_{1, 1}, \dots, \alpha_{1, n_1-1+c_1}, \alpha_{2, 1}, \dots, \alpha_{2, n_2-1}]$ since there are no crossings in $\id_{n_2-1} \sqcup \beta_1$ between the first $n_2$ strands. It follows that $C_{\Ras}(\id_{n_2-1} \sqcup \beta_1) = C_{\Ras}(\beta_1) \otimes_\F \F[\alpha_{2, 1}, \dots, \alpha_{2, n_2-1}]$. Similar formulas are obtained for $Q_{\beta_2 \sqcup \id_{n_1-1}}$ and $C_{\Ras}(\beta_2 \sqcup \id_{n_1-1})$. Finally, we can compute $C_{\Ras}(\beta)$ as a tensor product of Rasmussen's complexes by Definition \ref{def:tensor product of Rasmussen's complexes}. We have
\begin{align*}
C_{\Ras}(\beta) &= (C_{\Ras}(\id_{n_2-1} \sqcup \beta_1) \otimes_{Q_{\id_{n_2-1} \sqcup \beta_1}} Q_\beta) \otimes_{Q_\beta} (C_{\Ras}(\beta_2 \sqcup \id_{n_1-1}) \otimes_{Q_{\beta_2 \sqcup \id_{n_1-1}}} Q_\beta)\\
&=(C_{\Ras}(\beta_1) \otimes_\F \F[\alpha_{2, 1}, \dots, \alpha_{2, n_2-1}] \otimes_{\F[\alpha_{1, 1}, \dots, \alpha_{1, n_1-1+c_1}, \alpha_{2, 1}, \dots, \alpha_{2, n_2-1}]} Q_\beta) \otimes_{Q_\beta}\\
& \quad \quad (C_{\Ras}(\beta_2) \otimes_\F \F[\alpha_{1, 1}, \dots, \alpha_{1, n_1-1}] \otimes_{\F[\alpha_{2, 1}, \dots, \alpha_{2, n_2-1+c_2}, \alpha_{1, 1}, \dots, \alpha_{1, 1_2-1}]} Q_\beta)\\
& \cong (C_{\Ras}(\beta_1) \otimes_{\F} \F[\alpha_{2, 1}, \dots, \alpha_{2, n_2-1+c_2}]) \otimes_{Q_\beta} (C_{\Ras}(\beta_2) \otimes_{\F} \F[\alpha_{1, 1}, \dots, \alpha_{1, n_1-1+c_1}])\\
& \cong C_{\Ras}(\beta_1) \otimes_\F C_{\Ras}(\beta_2).
\end{align*}
Taking homology with respect to the Hochschild and Rasmussen's differentials by using the K\"unneth theorem establishes the isomorphism on homology.
\end{proof}
\begin{remark}
Note that the proof of the theorem establishes the isomorphism on the level of chain complexes while it only uses their formal properties -- the details of the differentials are not important. As such, the proof is indifferent to the presence of extra structure and moreover shows that connected sums commute with both Rasmussen's spectral sequences from Section \ref{sec:higher differentials} and dot-sliding homotopies from Section \ref{sec:dot-sliding homotopies}.
\end{remark}

\bibliography{mybib.bib}
\bibliographystyle{alpha}
\end{document}